\documentclass[12pt, twoside, leqno]{article}

\usepackage{amsmath,amsthm}
\usepackage{amssymb}

\usepackage{enumitem}

\usepackage{graphicx}
\usepackage{tikz-cd}

\usepackage[T1]{fontenc}

\newtheorem{theorem}{Theorem}[section]

\newtheorem{lemma}[theorem]{Lemma}
\newtheorem{proposition}[theorem]{Proposition}

\theoremstyle{definition}
\newtheorem{definition}[theorem]{Definition}
\newtheorem{remark}[theorem]{Remark}
\newtheorem{example}[theorem]{Example}

\numberwithin{equation}{section}

\newcommand\C{\mathbb{C}}

\newcommand\N{\mathbb{N}}
\newcommand\B{\mathcal{B}}

\newcommand\op{{\operatorname{op}}}
\newcommand\Hom{\operatorname{Hom}}

\newcommand\Aut{\operatorname{Aut}}
\newcommand\id{\operatorname{id}}

\newcommand\Baire{\mathcal{B}a}

\newcommand\AbsMes{\mathbf{AbsMbl}}
\newcommand\ConcMes{\mathbf{CncMbl}}
\newcommand\AbsProb{\mathbf{AbsPrb}}
\newcommand\CH{\mathbf{CH}}
\newcommand\Cat{\mathcal{C}}

\newcommand\CHProb{{\mathbf{CHPrb}}}
\newcommand\Bool{\mathbf{Bool}}
\newcommand\ProbAlg{{\mathbf{PrbAlg}}}
\newcommand\ProbAlgG{{\mathbf{PrbAlg}_\Gamma}}
\newcommand\ProbAlgGK{{\mathbf{PrbAlg}_{\Gamma\times K^\op}}}
\newcommand\SigmaAlg{{\mathbf{Bool}_\sigma}}
\newcommand\Set{\mathbf{Set}}
\newcommand\ConcProb{\mathbf{CncPrb}}

\newcommand\Hilb{\mathbf{Hilb}}
\newcommand\OpHilb{{\mathbf{Hilb}^*}}
\newcommand\CHGroupGen{{\mathbf{CHGrp}}}
\newcommand\CHGroup{{\CHGroupGen_\pi}}

\newcommand\Mes{\mathtt{Alg}}
\newcommand\BaireFunc{\mathtt{Bair}}

\newcommand\Stone{\mathtt{Conc}}
\newcommand\Inc{{\mathtt{Inc}}}
\newcommand\Abs{\mathtt{Abs}}
\newcommand\ident{\mathtt{id}}
\newcommand\Func{\mathtt{Func}}

\newcommand\Vertex{\mathtt{Vertex}}

\newcommand\Cond{\mathtt{Cond}}
\newcommand\CondTilde{{\widetilde{\mathtt{Cond}}}}
\newcommand\Forget{\mathtt{Forget}}
\newcommand\Cast{\mathtt{Cast}}
\newcommand\Inv{\mathtt{Inv}_\Gamma}
\newcommand\Haar{\mathrm{Haar}}

\begin{document}


\baselineskip=17pt


\title{An uncountable Mackey--Zimmer theorem}

\author{Asgar Jamneshan\\
Ko\c{c} University\\
Department of Mathematics\\ 
34450 Istanbul, Turkey  \\
Email: ajamneshan@ku.edu.tr \\
\and 
 Terence Tao\\
 University of California, Los Angeles \\
 Department of Mathematics\\
 Los Angeles,  CA 90095-1555, USA \\
 Email: tao@math.ucla.edu}

\date{}

\maketitle


\renewcommand{\thefootnote}{}

\footnote{2020 \emph{Mathematics Subject Classification}: Primary 37A15; Secondary 37A20.}

\footnote{\emph{Key words and phrases}: Mackey--Zimmer theorem, Mackey theory, ergodic group extensions, uncountable ergodic theory, point-free measure theory.}

\renewcommand{\thefootnote}{\arabic{footnote}}
\setcounter{footnote}{0}


\begin{abstract}
The Mackey--Zimmer theorem classifies ergodic group extensions $X$ of a measure-preserving system $Y$ by a compact group $K$, by showing that such extensions are isomorphic to a group skew-product $X \equiv Y \rtimes_\rho H$ for some closed subgroup $H$ of $K$.  An analogous theorem is also available for ergodic homogeneous extensions $X$ of $Y$, namely that they are isomorphic to a homogeneous skew-product $Y \rtimes_\rho H/M$. These theorems have many uses in ergodic theory, for instance playing a key role in the Host--Kra structural theory of characteristic factors of measure-preserving systems.

The existing proofs of the Mackey--Zimmer theorem require various "countability", "separability", or "metrizability" hypotheses on the group $\Gamma$ that acts on the system, the base space $Y$, and the group $K$ used to perform the extension.  In this paper we generalize the Mackey--Zimmer theorem to "uncountable" settings in which these hypotheses are omitted, at the cost of making the notion of a measure-preserving system and a group extension more abstract. However, this abstraction is partially counteracted by the use of a "canonical model" for abstract measure-preserving systems developed in a companion paper. 
In subsequent work we will apply this theorem to also obtain uncountable versions of the Host--Kra structural theory.
\end{abstract}

\section{Introduction}
\label{intro}

The purpose of this paper is to extend the Mackey--Zimmer theorem classifying ergodic group extensions to the "uncountable" setting in which the group acting is not required to be countable, and the spaces the group is acting on are not required to be separable or metrizable.  

\subsection{The countable Mackey--Zimmer theorem}

In this section we review the classical "countable" theorem of Mackey \cite{mackey} and Zimmer \cite{zimmer1976extension}.  It will be convenient to formulate the theorem here using the language of category theory, in order to utilize some foundational material developed in a companion paper \cite{jt-foundational}.  A review of the category-theoretic notation we employ here can be found in \cite[Appendix A]{jt-foundational}.

We will need a certain amount of notation.  We begin by recalling two categories of probability spaces from our companion paper \cite{jt-foundational}: the familiar category $\ConcProb$ of concrete probability spaces, and the somewhat less familiar category $\ProbAlg$ of   probability algebras.  (Two other categories $\CHProb$, $\AbsProb$ of probability spaces will be introduced later.)
F
\begin{definition}[Probability categories]\
\begin{itemize}
\item[(i)] \cite[Definition 5.1(ii)]{jt-foundational} A \emph{$\ConcProb$-space} (or \emph{concrete probability space}) is a triplet 
$X = (X_\Set, \Sigma_X, \mu_X)$, where $X_\Set$ is a set, $\Sigma_X$ is a $\sigma$-algebra on $X_\Set$, and $\mu_X \colon \Sigma_X \to [0,1]$ is a countably additive probability measure.  A \emph{$\ConcProb$-morphism} $T \colon X \to Y$ between two $\ConcProb$-spaces $X = (X_\Set, \Sigma_X, \mu_X)$, $Y = (Y_\Set,\Sigma_Y, \mu_Y)$ is a measurable map $T_\Set \colon X_\Set \to Y_\Set$ (with associated pullback map $T_\SigmaAlg \colon \Sigma_Y \to \Sigma_X$ defined by $T_\SigmaAlg(E) := T_\Set^{-1}(E)$) such that $T_* \mu_X = \mu_Y$, where $T_* \mu_X$ is the pushforward measure $T_* \mu_X := \mu_X \circ T_\SigmaAlg$.  Composition of $\ConcProb$-morphisms is given by the $\Set$-composition law.
\item[(ii)] \cite[Definition 6.1(vi)]{jt-foundational} A \emph{$\ProbAlg$-space} (or \emph{probability algebra}) is a pair $X = (\Inc(X)_\SigmaAlg, \mu_X)$, where\footnote{Probability algebras are the special case of \emph{measure algebras} in which the total measure is equal to one.  The symbol $\Inc$ denotes the inclusion functor from $\ProbAlg$ to the category $\AbsProb$ of abstract probability spaces, which are similar to  probability algebras but in which non-zero elements $E$ of the $\sigma$-algebra of zero measure are permitted; see Definition \ref{basic-func}.} $\Inc(X)_\SigmaAlg$ is a $\sigma$-complete Boolean algebra, and $\mu_X \colon \Inc(X)_\SigmaAlg \to [0,1]$ is a countably additive (abstract) probability measure such that $\mu_X(E)>0$ whenever $E \in \Inc(X)_\SigmaAlg$ is non-zero.
A \emph{$\ProbAlg$-morphism} $T \colon X \to Y$ between two $\ProbAlg$-spaces $X = (\Inc(X)_\SigmaAlg, \mu_X)$,
$Y = (\Inc(Y)_\SigmaAlg, \mu_Y)$ is a Boolean algebra homomorphism\footnote{By convention, we implicitly define $\ProbAlg$ as an opposite category in which the direction of all morphism arrows are reversed. This is to keep certain functors covariant, where we follow a common convention in category theory that functors are covariant by definition, and a contravariant functor is a synonym for a  functor on the opposite category.} $T^* = \Inc(T)_\SigmaAlg \colon \Inc(Y)_\SigmaAlg \to \Inc(X)_\SigmaAlg$ with the property that $\mu_Y = T_* \mu_X$, where $T_* \mu_X$ is the pushforward measure $T_* \mu_X := \mu_X \circ \Inc(T)_\SigmaAlg$.  $\ProbAlg$-composition is given by the law
$$ \Inc(S \circ T)_\SigmaAlg = \Inc(T)_\SigmaAlg \circ \Inc(S)_\SigmaAlg.$$
\end{itemize}
\end{definition}

\begin{remark} While the definition given above for a $\ProbAlg$-morphisms only requires $\Inc(T)_\SigmaAlg$
 to be Boolean homomorphisms, these homomorphisms in fact automatically preserve countable joins and meets.  To verify this claim, it suffices to show that $\bigwedge_{n \in \N} \Inc(T)_\SigmaAlg(E_n)$ vanishes whenever $E_n$ is a decreasing sequence in $\Inc(Y)_\SigmaAlg$ with $\bigwedge_{n \in \N} E_n = 0$. But from countable additivity we see that $\mu_X( E_n ) \to 0$, hence $\mu_Y( \Inc(T)_\SigmaAlg(E_n) ) \to 0$.  Since the $\Inc(T)_\SigmaAlg(E_n)$ are decreasing, the claim follows. As such, this definition of $\ProbAlg$ agrees with the one given in \cite{jt-foundational}.
\end{remark}

To every $\ConcProb$-space $X = (X_\Set, \Sigma_X, \mu_X)$ one can form an associated $\ProbAlg$-space $X_\ProbAlg = (\Inc(X_\ProbAlg)_\SigmaAlg, \mu_{X_\ProbAlg})$ by declaring $\Inc(X_\ProbAlg)_\SigmaAlg$ to be the $\sigma$-algebra $\Sigma_X$ quotiented by the $\sigma$-ideal of null sets, and defining $\mu_{X_\ProbAlg}$ to be the associated descent of $\mu_X$ to this quotient.  Informally, one can view $X_\ProbAlg$ as an abstraction of $X$ in which the null sets have been "deleted".  Any $\ConcProb$-morphism $T \colon X \to Y$ similarly induces an associated $\ProbAlg$-morphism $T_\ProbAlg \colon X_\ProbAlg \to Y_\ProbAlg$.  As an informal first approximation, $T_\ProbAlg$ is an abstraction of $T$ "up to almost everywhere equivalence", although one should not take this interpretation too literally in "uncountable" settings; see \cite[Example 5.2]{jt19}.  This operation of \emph{casting} $\ConcProb$ to $\ProbAlg$ is an example of what we call a \emph{casting functor}, and in this particular case can be factored as the composition $\Mes \circ \Abs$ of two other casting functors $\Mes, \Abs$, as depicted in the diagram of categories in Figure \ref{fig:categories}.

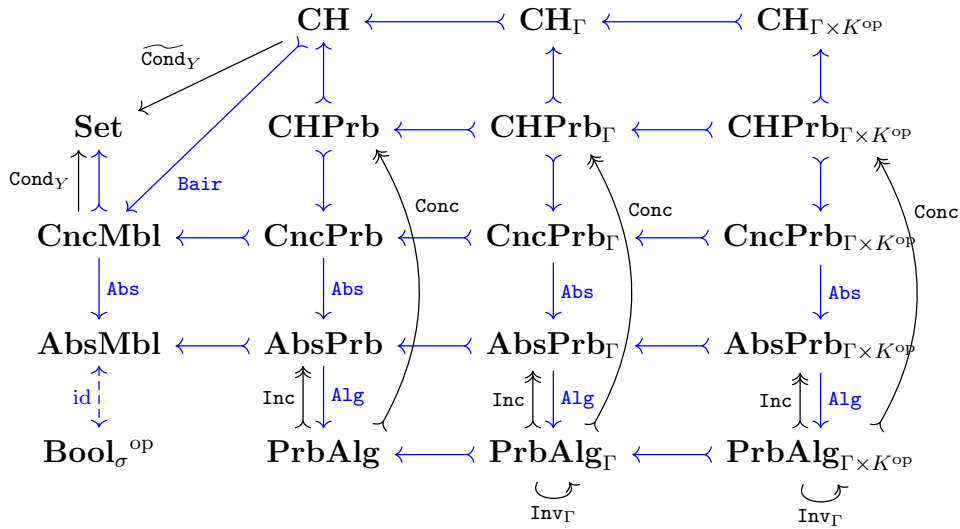
\begin{figure}
    \centering
    \begin{tikzcd}
    &\CH \arrow[ddl, blue, tail, "\BaireFunc", near end] \arrow[dl, "\CondTilde_Y"'] & \CH_\Gamma \arrow[l, tail,blue] & \CH_{\Gamma \times K^\op} \arrow[l, tail,blue] \\
    \Set  & \CHProb  \arrow[u, tail, blue] \arrow[d, tail, blue] & \CHProb_\Gamma \arrow[u, tail, blue]  \arrow[l, tail,blue] \arrow[d,tail,blue]  & \CHProb_{\Gamma \times K^\op} \arrow[u, tail, blue]  \arrow[l, tail,blue] \arrow[d,tail,blue] \\
     \ConcMes \arrow[d,blue, "\Abs"] \arrow[u,shift left = 1.5 ex, "\Cond_Y"]  \arrow[u, tail, blue] & \ConcProb \arrow[l, tail, blue] \arrow[d, blue, "\Abs"] & \ConcProb_\Gamma \arrow[l, tail, blue] \arrow[d, blue, "\Abs"] & \ConcProb_{\Gamma \times K^\op} \arrow[l, tail, blue] \arrow[d, blue, "\Abs"] \\
     \AbsMes   \arrow[d,dashed, leftrightarrow,blue,"\id"']  & \AbsProb \arrow[l, tail, blue]  \arrow[d, blue, "\Mes"] & \AbsProb_\Gamma \arrow[l, tail, blue]  \arrow[d, blue,  "\Mes"]  & \AbsProb_{\Gamma \times K^\op} \arrow[l,tail, blue]  \arrow[d, blue, "\Mes"]  \\
      \SigmaAlg^\op & \ProbAlg \arrow[u, shift left = 1.5 ex, tail, two heads, "\Inc"] \arrow[uuu, bend right, shift right = 3ex,tail , two heads, "\Stone"' near end] & \ProbAlgG \arrow[l, tail, blue] \arrow[u, shift left = 1.5 ex, tail,two heads, "\Inc"] \arrow[loop, out=240,in=300, looseness = 2, "\Inv"'] \arrow[uuu, bend right, shift right = 3, tail, two heads, "\Stone"' near end] & \ProbAlgGK \arrow[l, tail, blue] \arrow[u, shift left = 1.5 ex, tail,two heads, "\Inc"]\arrow[loop, out=240,in=300, looseness = 2, "\Inv"'] \arrow[uuu, bend right, shift right = 3ex, tail, two heads, "\Stone"' near end] \\
    \end{tikzcd}
    \caption{The primary categories and functors used in this paper (the definitions are reviewed in Section \ref{categories-sec}).  Unlabeled arrows refer to forgetful functors. Arrows with tails are faithful functors and arrows with two heads in one direction are full.  The diagram is not fully commutative (even modulo natural isomorphisms), but the functors in blue form a commuting subdiagram and will be used for casting operators $\Cast_{\Cat \to \Cat'}$ (see Definition \ref{cast}).  These conventions will be in force in all other diagrams of categories in this paper.}
    \label{fig:categories}
\end{figure} 

A key notion here will be that of \emph{isomorphism in $\ProbAlg$}: two $\ConcProb$-spaces $X,Y$ are isomorphic in $\ProbAlg$ if there is a $\ProbAlg$-isomorphism $T \colon X_\ProbAlg \to Y_\ProbAlg$ between their associated $\ProbAlg$-spaces $X_\ProbAlg$, $ Y_\ProbAlg$.  As an informal first approximation, isomorphism in $\ProbAlg$ asserts that $X$ and $Y$ are "equivalent modulo null sets", and is a strictly weaker notion than $\ConcProb$-isomorphism.  For instance, a $\ConcProb$-space is always isomorphic in $\ProbAlg$ to its measure-theoretic completion, but this completion need not be $\ConcProb$-isomorphic to the original space in general.

Now we add dynamics to these categories, using the following general construction, which also recently appeared implicitly in \cite{moriakov}.

\begin{definition}[Dynamical categories]\label{dynamic} Let $\Cat$ be a category (e.g., $\Cat=\ConcProb, \ProbAlg$), and let $\Gamma$ be a group.  We define the category $\Cat_\Gamma$ as follows:
\begin{itemize}
    \item[(i)]  A $\Cat_\Gamma$-object (or \emph{$\Cat_\Gamma$-system}) is a pair $X = (X_\Cat,T_X)$, where $X_\Cat$ is a $\Cat$-space and $T_X \colon \gamma \mapsto T_X^\gamma$ is a group homomorphism from $\Gamma$ to the automorphism group $\Aut_\Cat(X_\Cat)$ of $\Cat$-isomorphisms of $X_\Cat$.  We refer to $T_X$ as a  \emph{$\Cat$-action} of $\Gamma$ on $X$.  
    \item[(ii)] A $\Cat_\Gamma$-morphism $\pi \colon X \to Y$ from one $\Cat_\Gamma$-system $X=(X_\Cat,T_X)$ to another $Y=(Y_\Cat,T_Y)$ is a $\Cat$-morphism $\pi_\Cat \colon X \to Y$ such that the diagram
\begin{center}
    \begin{tikzcd}
    X_\Cat \arrow[r, "T_X^\gamma"] \arrow[d, "\pi_\Cat"] & X_\Cat \arrow[d, "\pi_\Cat"] \\
    Y_\Cat \arrow[r, "T_Y^\gamma"] & Y_\Cat
\end{tikzcd}
\end{center}
commutes in $\Cat$ for all $\gamma \in \Gamma$.
    \item[(iii)] Any functor $\Func \colon \Cat \to \Cat'$ induces a corresponding functor $\Func = \Func_\Gamma \colon \Cat_\Gamma \to \Cat'_\Gamma$ by mapping a $\Cat_\Gamma$-system $(X_\Cat,T_X)$ to $(\Func(X_\Cat), \Func(T_X))$, and mapping a $\Cat_\Gamma$-morphism $\pi \colon X \to Y$ to $\Func(\pi_\Cat)$ (which can be easily seen to be promoted from a $\Cat'$-morphism to a $\Cat'_\Gamma$-morphism).   
\end{itemize}
\end{definition}

A $\ConcProb_\Gamma$-system is also known as a \emph{(concrete) measure-preserving system} for the group $\Gamma$; a $\ProbAlgG$ is informally an abstraction of this concept in which all null sets have been "removed".  The casting functor $\Mes \circ \Abs$ from $\ConcProb$ to $\ProbAlg$ induces a casting functor $\Mes \circ \Abs = \Mes_\Gamma \circ \Abs_\Gamma$ from $\ConcProb_\Gamma$ to $\ProbAlgG$, associating a $\ProbAlgG$-system $X_{\ProbAlgG}$ to each $\ConcProb_\Gamma$-system $X$.  Then, as before, we can call two $\ConcProb_\Gamma$-systems \emph{isomorphic in $\ProbAlgG$} if there is a $\ProbAlgG$-isomorphism between the associated $\ProbAlgG$-systems.  Informally, this means that the two systems "agree up to null sets", with the caveat (which is important when the group $\Gamma$ is uncountable) that the null sets where the actions disagree are permitted to vary with the choice of group element.

Given a $\ConcProb_\Gamma$-system $Y$, a \emph{$\ConcProb_\Gamma$-extension} of $Y$ is a pair $(X,\pi)$, where $X$ is a $\ConcProb_\Gamma$-system $X$ and $\pi \colon X \to Y$ is a $\ConcProb_\Gamma$-morphism; we also call $(Y,\pi)$ a \emph{$\ConcProb_\Gamma$-factor} of $X$.  In both cases we refer to $\pi$ as the \emph{factor map}.  The collection $(\ConcProb_\Gamma \downarrow Y)$ of $\ConcProb_\Gamma$-extensions of $Y$ also forms a category, as part of the general construction of slice categories $(\Cat \downarrow D)$ of a category $\Cat$ over a $\Cat$-object; see \cite[Definition A.5]{jt-foundational}; a $(\ConcProb_\Gamma \downarrow Y)$-morphism $L$ from one $(\ConcProb_\Gamma \downarrow Y)$-space $(X,\pi)$ to another $(X',\pi')$ is then a $\ConcProb_\Gamma$-morphism $L \colon X \to X'$ such that $\pi' \circ L = \pi$. The collection $(X \downarrow \ConcProb_\Gamma)$ of $\ConcProb_\Gamma$-factors of $X$ is similarly a category (a special case of the coslice category in  \cite[Definition A.5]{jt-foundational}).

Given a $\ProbAlgG$-system $Y$, one can similarly define the category $(\ProbAlgG \downarrow Y)$ of 
$\ProbAlgG$-extensions of $Y$.  If $Y$ is a $\ConcProb_\Gamma$-system, there is an obvious casting functor from $(\ConcProb_\Gamma \downarrow Y)$ to $(\ProbAlgG \downarrow Y_{\ProbAlgG})$, and one can then define the notion of two $\ConcProb_\Gamma$-extensions in $(\ConcProb_\Gamma \downarrow Y)$ being isomorphic in $(\ProbAlgG \downarrow Y_{\ProbAlgG})$.  Again, this captures the informal notion of the two extensions being "equivalent up to null sets".  Similarly for factors instead of extensions.

\begin{remark}  $\ProbAlgG$-morphisms are essentially the same concept as "homomorphisms of measure preserving systems" in \cite[Definition 5.4]{furstenberg2014recurrence}, \cite[Definition 12.7]{EFHN} or "morphisms of measure-preserving dynamical systems" in \cite[\S 2.2]{glasner2015ergodic}, though in those references one restricts attention to the "countable" setting in which the $\sigma$-complete Boolean algebra is countably generated.  One additional minor difference is that in these references the ambient space is a $\ConcProb$-space (equipped with an action of $\Gamma$ on the corresponding $\ProbAlg$-space), rather than a $\ProbAlg$-space, but this difference is not of major relevance in applications since every $\ProbAlg$-space has at least one $\ConcProb$-space model.  See also \cite[Ch. VI]{DNP} where the notion of a $\ProbAlg$-isomorphism is essentially introduced (with the same technical difference as mentioned previously), and the observation made that such isomorphisms can also be viewed as  Banach lattice isomorphisms (or Markov isomorphisms) of the corresponding $L^1$ spaces.  See the final subsection of \cite[\S 12.3]{EFHN} for further discussion comparing the concrete and abstract approaches to ergodic theory.  One can also enlarge the category $\ProbAlg$ by replacing the class of $\ProbAlg$-morphisms with the more general class of Markov operators, which is the abstract analogue of the class of probability kernels on $\ConcProb$-spaces.  These categories are studied in \cite[Ch. 13]{EFHN} and \cite{voevodsky} respectively.  However, we will not need these larger categories in our current work.
\end{remark}

There is an \emph{invariant factor functor} $\Inv \colon \ProbAlgG \to \ProbAlgG$ that takes an  probability algebra system $X = (X_\ProbAlg, T_X)$ and returns its invariant factor; see Definition \ref{inv-def} for a precise definition.  A $\ProbAlgG$-system $X$ is said to be \emph{ergodic} if $\Inv(X)$ is trivial, and a $\ConcProb_\Gamma$-system $X$ is \emph{ergodic} if the associated $\ProbAlgG$-system $X_{\ProbAlgG}$ is ergodic.

The Mackey--Zimmer theorem concerns two key ways to extend a given $\ConcProb_\Gamma$-system $Y$:

\begin{definition}[Concrete skew-products and group extensions]\label{concrete-ext}  Let $Y = (Y_\ConcProb, T_Y)$ be a $\ConcProb_\Gamma$-system, let $K$ be a compact Hausdorff group\footnote{One could of course organize the compact Hausdorff groups into a category if desired.  In fact there are \emph{two} natural categories one could use here: the category $\CHGroupGen$ of compact Hausdorff groups with morphisms that are continuous homomorphisms, and the subcategory $\CHGroup$ in which the morphisms are also required to be surjective.  The latter category $\CHGroup$ interacts well with Haar measure (surjective continuous homomorphisms preserve Haar measure), allowing one to interpret the Haar measure construction as a functor from $\CHGroup$ to $\CHProb$.  However we will not need to extensively use the category theoretic properties of compact Hausdorff groups in this paper.}, and let $L$ be a compact subgroup of $K$.  We endow $K$ (resp. $K/L$) with the Baire $\sigma$-algebra\footnote{The Baire algebra of a compact Hausdorff space $K$ is the $\sigma$-algebra generated by the space $C(K)$ of continuous functions of $K$ into $\C$; equivalently, it is the $\sigma$-algebra generated by compact $G_\delta$ subsets of $K$.  For metrizable $K$ the Baire algebra agrees with the Borel algebra, but for non-metrizable $K$ the Baire algebra can be strictly smaller.  See \cite{jt-foundational} for more discussion of why the Baire algebra is preferred over the Borel algebra in uncountable analysis.}, as well as the bi-invariant Haar probability measure $\Haar_K$ (resp. the left-invariant Haar probability measure $\Haar_{K/L}$). 
\begin{itemize}
    \item[(i)]  A \emph{$K$-valued $\ConcProb_\Gamma$-cocycle} on $Y$ is a tuple $\rho = (\rho_\gamma)_{\gamma \in \Gamma}$ of measurable maps $\rho_\gamma \colon Y_\ConcProb \to K$ that obeys the \emph{cocycle equation}
    \begin{equation}\label{cocycle-eq} \rho_{\gamma  \gamma'}(y) = \rho_\gamma(T_Y^{\gamma'} y) \rho_{\gamma'}(y)
    \end{equation}
    for all $\gamma,\gamma' \in \Gamma$ and $y \in Y_\Set$.
    \item[(ii)] If $\rho$ is a $K$-valued $\ConcProb_\Gamma$-cocycle, we define the \emph{$\ConcProb_\Gamma$-homogeneous skew-product} 
    $$Y \rtimes^{\ConcProb_\Gamma}_\rho K/L = Y \rtimes_\rho K/L = (X, \pi) \in (\ConcProb_\Gamma \downarrow Y)$$
    by taking the product space $X_\ConcProb = Y_\ConcProb \times^{\ConcProb} K/L$ (using the standard product construction in $\ConcProb$), and defining the action $T_X$ on $X_\ConcProb$ by the formula
    $$ T_X^\gamma(y,kL) := (T_Y^\gamma(y), \rho_\gamma(y) kL).$$
    for $(y,kL) \in X_\Set$, with the projection map $\pi = \pi_{X \to Y} \colon X \to Y$ defined by $\pi(y,kL) = y$ for $(y,kL) \in X_\Set$.  If $L=\{1\}$, we refer to $Y \rtimes^{\ConcProb_\Gamma} K$ as a \emph{$\ConcProb_\Gamma$-group skew-product}.
    \item[(iii)]  A \emph{$\ConcProb_\Gamma$-homogeneous extension} of $Y$ by $K/L$ is a tuple $X = (X, \pi, \theta, \rho)$, where $(X, \pi)$ is an extension of $Y$, the \emph{vertical coordinate} $\theta \colon X_\ConcMes \to K/L$ is a measurable  function such that $\pi, \theta$ jointly generate the $\sigma$-algebra $\Sigma_X$ of $X$, and $\rho = (\rho_\gamma)_{\gamma \in \Gamma}$ is a $K$-valued $\ConcProb_\Gamma$-cocycle on $Y$ such that
    \begin{equation}\label{theta-gam-concrete}
    \theta(T_X^\gamma x) = \rho_\gamma(\pi(x)) \theta(x)
    \end{equation}
    for all $x \in X_\Set$ and $\gamma \in \Gamma$, using of course the left action of $K$ on $K/L$.  If $L=\{1\}$, we refer to $X$ as a \emph{$\ConcProb_\Gamma$-group extension}.  We sometimes write $\pi = \pi_{X \to Y}$ when we need to emphasize the domain and codomain of $\pi$.
\end{itemize}
\end{definition}

Homogeneous extensions are very closely related to the notions of \emph{isometric extensions} and \emph{compact extensions}, which play a fundamental role in the Furstenberg--Zimmer structure theory \cite{furstenberg2014recurrence, zimmer1976ergodic, zimmer1976extension} of measure-preserving systems, as well as subsequent refinements of that theory, such as is found in the work of Host and Kra \cite{host2005nonconventional}.  See \cite{jamneshan2019fz} for a discussion of these relationships in both the countable and uncountable settings.  It is common in the literature to reduce to the \emph{corefree} case when $\bigcap_{k \in K} kLk^{-1} = \{1\}$, by quotienting out by the normal core $\bigcap_{k \in K} kLk^{-1}$, but we will not need to use the corefree property here.

One can easily check using the cocycle equation \eqref{cocycle-eq} (and the Fubini--Tonelli theorem and invariance of Haar measure) that if $\rho$ is a $K$-valued $\ConcProb_\Gamma$-cocycle, then the  $\ConcProb_\Gamma$-homogeneous skew-product $Y \rtimes^{\ConcProb_\Gamma}_\rho K/L$ is indeed a $\ConcProb_\Gamma$-extension of $Y$, and is in fact a $\ConcProb_\Gamma$-homogeneous extension with the vertical coordinate $\theta(y,kL) := kL$.

The Mackey--Zimmer theorem \cite{mackey, zimmer1976extension} asserts a partial converse to this implication, namely that (under some countability hypotheses), an \emph{ergodic} $\ConcProb_\Gamma$-homogeneous extension of $Y$ is isomorphic (in $(\ProbAlgG \downarrow Y_{\ProbAlgG})$) to a $\ConcProb_\Gamma$-homogeneous skew-product, possibly after passing from $K$ to a subgroup.  The following formulation of the theorem is essentially in \cite{glasner2015ergodic}:

\begin{theorem}[Countable Mackey--Zimmer theorem]\label{mackey-countable}  Let $\Gamma$ be a group, let $Y$ be an ergodic $\ConcProb_\Gamma$-system, and let $K$ be a compact Hausdorff group.  Assume the following additional hypotheses:
\begin{itemize}
    \item[(a)] $\Gamma$ is at most countable.
    \item[(b)] $Y$ is a standard Lebesgue space (a standard Borel space equipped with a regular probability measure).
    \item[(c)] $K$ is metrizable.
\end{itemize}
Then
\begin{itemize}
    \item[(i)] Every ergodic $\ConcProb_\Gamma$-group extension $X$ of $Y$ by $K$ is isomorphic in $(\ProbAlgG \downarrow Y_{\ProbAlg})$ to a $\ConcProb_\Gamma$-group skew-product $Y \rtimes^{\ConcProb_\Gamma}_\rho H$ for some compact subgroup $H$ of $K$ and some $H$-valued $\ConcProb_\Gamma$-cocycle $\rho$.
    \item[(ii)] Every ergodic $\ConcProb_\Gamma$-homogeneous extension $X$ of $Y$ by $K/L$ for some compact subgroup $L$ of $K$ is isomorphic in $(\ProbAlgG \downarrow Y_{\ProbAlg})$ to a $\ConcProb_\Gamma$-homogeneous skew-product $Y \rtimes^{\ConcProb_\Gamma}_\rho H/M$ for some compact subgroup $H$  of $K$, some compact subgroup $M$ of $H$, and some $H$-valued $\ConcProb_\Gamma$-cocycle $\rho$.  
\end{itemize}
\end{theorem}

\begin{proof}
A proof of (i) can be found in \cite[Theorem 3.25(5)]{glasner2015ergodic}, while a proof of (ii) can be found in \cite[Theorem 3.26]{glasner2015ergodic}. 
\end{proof}

We refer to this theorem as a "countable" theorem because of the hypotheses (a), (b), (c) which place countability (or separability) type axioms on the data $\Gamma, Y, K$.  The objective of this paper is to eliminate these countability hypotheses from Theorem \ref{mackey-countable}, in order to make the theory applicable to "uncountable" settings in which $\Gamma$ could for instance be an ultraproduct of groups, $Y$ could be a Loeb space, and $K$ an uncountable product of compact groups.  A similar elimination was achieved by us in \cite{jt19} for the Moore--Schmidt theorem regarding the cohomology of cocycles.  In that result it was necessary to formulate the result exclusively in the abstract framework, and the same phenomenon occurs here.  More precisely, the main result of this paper is

\begin{theorem}[Uncountable Mackey--Zimmer theorem]\label{mackey-uncountable}  Let $\Gamma$ be a group, $Y$ be an ergodic $\ProbAlgG$-system, and $K$ be a compact Haudorff group.
\begin{itemize}
    \item[(i)] Every ergodic $\ProbAlgG$-group extension $X$ of $Y$ by $K$ is isomorphic in $(\ProbAlgG \downarrow Y)$ to a $\ProbAlgG$-group skew-product $Y \rtimes^{\ProbAlgG}_\rho H$ for some compact subgroup $H$ of $K$ and some $H$-valued $\ProbAlgG$-cocycle $\rho$. 
    \item[(ii)] Every ergodic $\ProbAlgG$-homogeneous extension $X$ of $Y$ by $K/L$ for some compact subgroup $L$ of $K$ is isomorphic in $(\ProbAlgG \downarrow Y)$ to a $\ProbAlgG$-homogeneous skew-product $Y \rtimes^{\ProbAlgG}_\rho H/M$ for some compact subgroup $H$ of $K$, some compact subgroup $M$ of $H$, and some $H$-valued $\ProbAlgG$-cocycle $\rho$.  
\end{itemize}
\end{theorem}

The notions of $\ProbAlgG$-group skew-products,
$\ProbAlgG$-homogeneous skew-products, $\ProbAlgG$-group extensions, and
$\ProbAlgG$-homogeneous extensions used in the above theorem will be defined formally in Definition \ref{abstract-ext}; for now, it will suffice to say that they are the natural analogues of their concrete counterparts in Definition \ref{concrete-ext}.

The implication of Theorem \ref{mackey-countable} from Theorem \ref{mackey-uncountable} is almost immediate, but requires a verification that the abstraction of concrete skew-products does not depend on the choice of model.  We give the details of this implication in Section \ref{implication}.  We remark that one additional advantage of the abstract formulation in Theorem \ref{mackey-uncountable} is that it also applies to \emph{near-cocycles} (which only obey the cocycle equation almost everywhere, rather than everywhere), which at the concrete level only generate \emph{near-actions} (in the sense of Zimmer \cite{zimmer1976ergodic}) rather than genuine actions; see for instance our recent article \cite{jt21} for an application of this variant of the Mackey--Zimmer theorem.

Part (i) of Theorem \ref{mackey-uncountable} was also previously established by Ellis \cite{ellis}; see Appendix \ref{appendix} for a proof of the equivalence of Ellis's results and part (i) of Theorem \ref{mackey-uncountable}. 
While there are many common elements between both proofs (most notably the reliance on canonical models, as well as a lifting lemma of Gleason \cite{gleason}), the arguments in \cite{ellis} are mostly from topological dynamics, whereas our approach is more ergodic theoretic in nature.  For instance, in \cite{ellis} the Mackey range $H$ is located via a Zorn's lemma argument based on locating a minimal subflow of a topological group extension associated to $X$, whereas in our arguments the Mackey range is described in more explicit ergodic-theoretic terms as the stabilizer of the $K^{\op}$-action on the $\Gamma$-invariant factor of $Y \rtimes_\rho K$.  In a similar vein, to establish that the measure on $X$ agrees (after suitable changes of variable) with product measure on $Y \times H$, the arguments in \cite{ellis} rely crucially on a non-trivial theorem of Keynes and Newton \cite{keynes-newton} describing the Choquet theory of invariant measures on compact group extensions, while our argument proceeds by a calculation (see Lemma \ref{integral-comp}) based on Fubini's theorem and the invariance properties of Haar measure.

\subsection{Proof methods}\label{proof-method}

To prove Theorem \ref{mackey-uncountable} we broadly follow the arguments in \cite[\S 3.5]{glasner2015ergodic} (as adapted in a blog post of the second author \cite{tao2014mackey}).  

\begin{figure}
    \centering
    \begin{tikzcd}
    K & X \arrow[l,dotted,"\theta"] \arrow[dl,dotted,"\theta_*"] \arrow[d,"\pi"] & X \rtimes_1 K \arrow[l,"\pi"]  \arrow[d,Rightarrow,"\Pi"] \arrow[r,Rightarrow,"\pi"] & \Inv(X \rtimes_1 K) \equiv K \arrow[d, Rightarrow,"\pi"] \\
    H & Y \arrow[d,dotted,"\psi_0"] \arrow[dl, dotted,"\Phi"] & Y \rtimes_\rho K \arrow[l,"\pi"] \arrow[r, Rightarrow, "\psi"] & \Inv(Y \rtimes_\rho K) \equiv H \backslash K \\
    K \arrow[r,dotted,"\pi"'] & H \backslash K
    \end{tikzcd}
    \caption{The commutative diagram to chase to prove part (i) of the Mackey--Zimmer theorem. Dotted arrows are merely measurable; ordinary arrows preserve the probability measure and the $\Gamma$ action; and the thick arrows also preserve the larger $\Gamma \times K^\op$ action. Subscripts on the various projection maps $\pi$ have been omitted for brevity, as has any mention of the inclusion functor $\Inc$.}
    \label{fig:diagram}
\end{figure}

For part (i) of the Mackey--Zimmer theorem, the given data consists of a group extension $(X,\pi)$ of $Y$ by a group $K$, together with a vertical coordinate $\theta \colon X \to K$.  The proof proceeds by expanding this diagram to Figure \ref{fig:diagram}, according to the following strategy (where for sake of discussion we ignore the technical distinctions between concrete and abstract categories).

\begin{enumerate}
    \item First, one lifts the extension $(X, \pi)$ of the $\Gamma$-system $Y$ to an extension $(X \rtimes_1 K, \Pi)$ of the $\Gamma \times K^\op$-system $Y \rtimes_\rho K$ defined by $\Pi(x,k):= (\pi(x), \theta(x) k)$, where $\rho$ is the cocycle associated to the group extension $X$, and $X \rtimes_1 K$ is the group skew-extension of $X$ by $K$ using the trivial cocycle $1$ (this is the same as the direct product $X \times K$ of $X$ and $K$, where $\Gamma$ acts trivially on $K$ and $K^\op$ acts trivially on $X$).  The point is that this new system also acquires a right action of $K$ (or equivalently, a left action of the opposite group $K^\op$), which commutes with the existing $\Gamma$ action to promote this extension to an extension of $\Gamma \times K^{\op}$-systems.
    \item Using the ergodicity of $X$, one can show that the invariant factor $\Inv(X \rtimes_1 K)$ is isomorphic to $K$, which then implies that $\Inv(Y \rtimes_\rho K)$ is a factor of $K$.  All of these factors remain preserved by the right action of $K$. Using an argument based on the Stone--Weierstra{\ss} theorem and the use of convolutions to approximate measurable functions by continuous ones, one can classify the factors of $K$ as being of the form $H\backslash K$ (up to isomorphism) for some closed subgroup $H$ of $K$.  Thus $\Inv(Y \rtimes_\rho K)$ is isomorphic to $H\backslash K$ for some such $H$ (known as the \emph{Mackey range} of $\rho$).
    \item The factor map $\psi \colon Y \times K \to H \backslash K$ is equivariant with respect to the right-action of $K$, and is thus determined by a map $\psi_0 \colon Y \to H \backslash K$.  If we lift this map to a map $\Phi \colon Y \to K$, one can use this map to "straighten" the vertical coordinate $\theta$, to obtain a new vertical coordinate $\theta_*$ that takes values in $H$ rather than $K$.
    \item By appealing to the Fubini--Tonelli theorem, and the translation invariance of Haar measure, one can show that the pair $(\pi,\theta_*)$ pushes the measure $\mu_X$ on $X$ to the product of the measure $\mu_Y$ on $Y$ and Haar measure $\Haar_H$ on $H$, at which point it is straightforward to show that $X$ is equivalent to a skew-product $Y \rtimes_{\rho_*} H$ of $Y$ by $H$.
\end{enumerate}

For part (ii), the given data now consists of a homogeneous extension $(X,\pi)$ of $Y$ by a group quotient $K/L$, together with a vertical coordinate $\theta \colon X \to K/L$.  The proof now proceeds by expanding the diagram to Figure \ref{fig:diagram2}, according to the following strategy:

\begin{figure}
    \centering
    \begin{tikzcd}
    H \arrow[d, "\pi"] & X' \arrow[l,"\theta'_*"'] \arrow[d,"\pi"] \arrow[r,dotted,"\theta'"] & K \arrow[d, "\pi"] \\
    H/M & X \arrow[l,"\theta_*"'] \arrow[d, "\pi"] \arrow[r, dotted, "\theta"] & K/L \\
    & Y
    \end{tikzcd}
    \caption{The commutative diagram to chase to prove part (ii) of the Mackey--Zimmer theorem. Dotted arrows are merely measurable; ordinary arrows preserve the probability measure and the $\Gamma$ action. Subscripts on the various projection maps $\pi$ have been omitted for brevity.}
    \label{fig:diagram2}
\end{figure}

\begin{enumerate}
    \item First, one lifts the homogeneous extension $X$ of $Y$ by $K/L$ to a group extension $X$ of $Y$ by $K$.  It is not difficult to locate a non-ergodic extension of this type by an explicit construction; but in order to apply part (i) we will need the extension to be ergodic.  This can be accomplished by an argument involving the Riesz representation theorem and the Krein--Milman theorem.
    \item Applying part (i), we can view $X'$ as a group skew-product of $Y$ with some closed subgroup $H$ of $K$, thus imbuing $X'$ with a vertical coordinate $\theta'_*$ which is now measure-preserving.  The original vertical coordinate $\theta$ of $X$ then also gets associated with a corresponding vertical coordinate $\theta_*$ in the quotient space $H/M$, where $M := H \cap L$.
    \item From Fubini's theorem one can verify that the pair $(\pi, \theta_*)$ pushes forward the measure $\mu_X$ on $X$ to the product of the measure $\mu_Y$ on $Y$ and the Haar measure $\Haar_{H/M}$ on $H/M$,  at which point it is straightforward to show that $X$ is equivalent to a skew-product $Y \rtimes_{\rho_*} H/M$ of $Y$ by $H/M$.
\end{enumerate}

When extending these arguments to uncountable settings, several steps need to be taken to overcome the additional technical difficulties that arise in this case:

\begin{itemize}
    \item The Nedoma pathology \cite{nedoma1957note} shows that the Borel $\sigma$-algebra behaves badly with respect to products on compact Hausdorff spaces that are not assumed to be metrizable.  In particular the group operations on a compact Hausdorff group can fail to be Borel measurable.  To avoid this problem, we follow \cite{jt19}, \cite{jt-foundational} and endow compact Hausdorff spaces with the Baire $\sigma$-algebra instead of the Borel $\sigma$-algebra, which is much better behaved. 
    \item To avoid having to take the uncountable union of null sets, we will again follow \cite{jt19}, \cite{jt-foundational} and pass from concrete measurable spaces, probability spaces, and measure-preserving systems to their abstract counterparts, which gains us the ability to "delete" the ideal of null sets so that this difficulty no longer arises.
    \item As is common in ergodic theory, it is convenient to model the  probability algebras that arise from the previous considerations by concrete models, in order to use "pointwise" tools such as the Fubini--Tonelli theorem, or pointwise manipulation of cocycles.  We use \emph{canonical model} appearing in various forms in the literature \cite{segal}, \cite{DNP},  \cite{fremlinvol3}, \cite{doob-ratio}, \cite{ellis}, \cite{EFHN}, \cite{jt-foundational} that models  probability algebras by concrete (and compact) probability spaces (somewhat analogously to how the Stone-{\v C}ech compactification "models" a locally compact Hausdorff space by a  compact space).  Crucially, the canonical model is functorial, so that it can also be used to model $\ProbAlgG$-systems by a special type of $\ConcProb_\Gamma$-system. See Theorem \ref{canon}. 
    \item The group $\Gamma$ is not assumed to be countable or amenable, and so many standard ergodic theorems no longer apply in this setting.  Fortunately, one can structure the argument in such a way that no ergodic theorems are required\footnote{In an earlier version of the paper we used an abstract ergodic theorem of Alaoglu and Birkhoff \cite{ab40abstract} that is valid for arbitrary group actions; we thank the referee for pointing out that even this ergodic theorem is not needed in the argument.}.
    \item In the uncountable setting, standard measurable selection lemmas no longer apply, which potentially causes difficulty in the step where one wishes to lift the map $\psi_0 \colon Y \to H \backslash K$ to a map $\Phi \colon Y \to K$ in an (abstractly) measurable fashion.  However it turns out to be possible to proceed by using a lifting result of Gleason \cite{gleason}, exploiting the extremal disconnectedness of (the canonical model of) $Y$.
\end{itemize}

As a variant of our main theorem (and following a suggestion of the anonymous referee), we also give a cocycle-free description of group extensions, in the spirit of \cite[Theorem 3.29]{glasner2015ergodic}; see Section \ref{cocycle-free}.

\subsection{Notation}

We will use the language of category theory throughout this paper.  For basic definitions, such as that of a category, functor, or natural transformation, see \cite[Appendix A]{jt-foundational} and the references therein.  We highlight in particular the \emph{casting convention} we will use to "automatically" convert objects or morphisms in one category $\Cat$ into another $\Cat'$ (cf.~\cite[Section 1.6]{jt-foundational}):

\begin{definition}[Casting operators]\label{cast}  Define a \emph{casting functor} (or \emph{casting operator}) to be any one of the following functors:
\begin{itemize}
    \item[(i)]  A functor depicted in blue in Figure \ref{fig:categories}. (In particular, all the forgetful functors in Figure \ref{fig:categories} are casting functors.)
    \item[(ii)]  The identity functor $\ident_\Cat$ on any category $\Cat$. 
    \item[(iii)]The vertex functors ${ \Vertex} \colon (\Cat \downarrow X) { \to} \Cat$, ${ \Vertex} \colon (X \downarrow \Cat) \to \Cat$ that map a $(\Cat \downarrow X)$-object $Y\to X$ or $(X\downarrow \Cat)$-object $X\to Y$ to its vertex object $X_\Cat$, and any morphism $f$ in $(\Cat \downarrow X)$ or $(X \downarrow \Cat)$ to the corresponding $\Cat$-morphism $f_\Cat$.  (For example, we have vertex functors from $(\Cat \downarrow Y)$ to $\Cat$ when $\Cat = \ProbAlg$, $\ConcProb$, $\ProbAlgG$, $\ConcProb_\Gamma$ and $Y$ is a $\Cat$-object.)
    \item[(iv)]  The obvious forgetful functor from $\Cat_{\Gamma'}$ to $\Cat_\Gamma$ whenever $\Gamma$ is a normal subgroup of $\Gamma'$.
    \item[(v)] Any finite composition of functors from the above list.
\end{itemize}
The casting functors in this paper are chosen to form a commutative diagram; thus for any two categories $\Cat,\Cat'$ there is at most one casting functor $\Cast_{\Cat \to \Cat'} \colon \Cat {\to} \Cat'$ from the former to the latter.  If such a casting functor exists, we say that $\Cat$ can be \emph{casted} to $\Cat'$, and for any $\Cat$-object $X = X_\Cat$ we define the \emph{cast} of $X$ to $\Cat'$ to be the corresponding object in $\Cat'$, we write $X_{\Cat'}$ for $\Cast_{\Cat \to \Cat'}(X)$, and refer to $X_{\Cat'}$ as the \emph{cast} of $X$ to $\Cat'$ (and $X_\Cat$ as a \emph{promotion} of $X_{\Cat'}$ to $\Cat$).  We may cast or promote $\Cat$-morphisms, $\Cat$-diagrams, products, coproducts, and tensors  in $\Cat$ to $\Cat'$ in a similar fashion.  Thus for instance a $\Cat'$-morphism has at most one promotion to a $\Cat$-morphism if the casting functor is faithful.  (Informally, one should view the $\Cat'$-cast of a mathematical structure associated to $\Cat$ as the "obvious" corresponding $\Cat'$-structure associated to the $\Cat$-structure, with the choice of casting functors in Figure \ref{fig:categories} in this paper formalizing what "obvious" means.)

When a mathematical expression or statement requires an object or morphism to lie in $\Cat$, but an object or morphism in another category $\Cat'$ appears in its place, then it is understood that a casting operator from $\Cat$ to $\Cat'$ is automatically applied.  In particular, if a statement is said to "hold in $\Cat'$" or "be interpreted in $\Cat'$", or if an object or morphism  is to be understood as a $\Cat'$-object or a $\Cat'$-morphism, then the appropriate casting operators to $\Cat$ are understood to be automatically applied.  We will sometimes write $X =_\Cat Y$ to denote the assertion that an identity $X=Y$ holds in $\Cat$.  

If one composes a named functor $\Func$ on the left or right (or both) with forgetful casting functors, the resulting functor will also be called $\Func$ when there is no chance of confusion (or if the ambiguity is irrelevant). 
\end{definition}

We give some examples to illustrate this casting convention (using concepts introduced in Section \ref{categories-sec} below).  Further examples may be found in \cite[Example A.23]{jt-foundational}.

\begin{example}\text{}
\begin{itemize}
\item[(i)] If $X = X_\CHProb = (X_\Set, {\mathcal F}_X, \mu_X)$ is a compact Hausdorff probability space, then $X_\Set$ is the associated underlying set, $X_\CH = (X_\Set, {\mathcal F}_X)$ is the associated compact Hausdorff space, $X_\ConcMes = \BaireFunc(X_\CH)=(X_\Set, \Sigma_{X_\Set})$ is the associated concrete measurable space where $X_\Set$ is equipped with the Baire $\sigma$-algebra $\Sigma_{X_\Set}=\Baire(X_\CH)$, $X_\ConcProb = (X_\ConcMes, \mu_X)$ is the associated concrete probability space, $X_\SigmaAlg = \Baire(X)$ is the Baire $\sigma$-algebra, $X_\ProbAlg = (\Baire(X)/{\mathcal N}_X, \mu_X/{\sim})$ is the associated  probability algebra, and $X_\AbsProb = (\Baire(X), \mu_X)$ is the associated abstract probability space.  In contrast, $\Inc(X) = \Inc(X_\ProbAlg) = (\Baire(X)/{\mathcal N}_X, \mu_X/{\sim})$ is a smaller abstract probability space than $X_\AbsProb$, in which the ideal of null sets has been "deleted".
\item[(ii)]  If $T \colon X \to Y$ is a $\ConcProb$-morphism, then $T_{\SigmaAlg} \colon Y_{\SigmaAlg} \to X_{\SigmaAlg}$ is the associated pullback map: $T_\SigmaAlg(E) = T^*(E) = T_\Set^{-1}(E)$.
\item[(iii)]  If $f_1, f_2 \colon X \to Y$ are $\ConcProb$-morphisms that agree almost everywhere, then they agree in $\ProbAlg$: $f_1 =_{\ProbAlg} f_2$, that is to say the $\ProbAlg$-morphisms $(f_1)_\ProbAlg \colon X_\ProbAlg \to Y_\ProbAlg$ and $(f_2)_{\ProbAlg} \colon X_\ProbAlg \to 
Y_\ProbAlg$ agree.  (The converse implication can fail; see \cite[Examples 5.1, 5.2, 5.3]{jt19}.)
\item[(iv)]  If $X = (X_\ConcMes, \mu_X)$, $Y = (Y_\ConcMes, \mu_Y)$ are $\ConcProb$-spaces, then a $\ConcMes$-morphism $T \colon X_\ConcMes \to Y_\ConcMes$ can be promoted to a $\ConcProb$-morphism $T_\ConcProb \colon X \to Y$ if and only if $T$ preserves the measure in the sense that $T_* \mu_X = \mu_Y$.  If this happens then the promotion is unique; this reflects the faithful nature of the casting functor from $\ConcProb$ to $\ConcMes$.
\end{itemize}
\end{example}

\section{Main categories and functors}\label{categories-sec}

In this section we present the categories and functors appearing in Figure \ref{fig:categories} that have not already been defined in the introduction.  Most of these categories and functors were already discussed in depth in the companion paper \cite{jt-foundational}; in those cases we only give an abbreviated description of these objects here, giving precise citations to locations in \cite{jt-foundational} that contain a more precise definition.

\subsection{Basic categories and functors}

We first review several categories and functors from \cite{jt-foundational}.

\begin{definition}[Basic categories]\
\begin{itemize}
\item[(i)] \cite[Example A.2]{jt-foundational}  $\Set$ is the category of sets $X$ (with morphisms being arbitrary functions $T \colon X \to Y$).
\item[(ii)]  \cite[Definition 6.1(ii)]{jt-foundational} $\SigmaAlg$ is the category of $\sigma$-complete Boolean algebras $\B$, with morphisms being Boolean homomorphisms $\phi \colon \B \to \B'$ preserving countable joins and meets.
\item[(iii)]  \cite[Definition 6.1(iii)]{jt-foundational} $\AbsMes$ is the opposite category to $\SigmaAlg$, thus for instance an $\AbsMes$-space (or \emph{abstract measurable space}) $X$ takes the form $X = X_\SigmaAlg$ for a $\SigmaAlg$-algebra $X$, and an $\AbsMes$-morphism $T \colon X \to Y$ takes the form $T = T_\SigmaAlg$ for a $\SigmaAlg$-morphism $T_\SigmaAlg \colon Y_\SigmaAlg \to X_\SigmaAlg$, and composition given by the law $S \circ T := (T_\SigmaAlg \circ S_\SigmaAlg)$.
\item[(iv)] \cite[Definition 3.1(iii)]{jt-foundational} $\ConcMes$ is the category of concrete measurable spaces $X$, with morphisms being measurable functions $T \colon X \to Y$.
\item[(v)]  \cite[Definition 6.1(iv)]{jt-foundational} $\AbsProb$ is the category of \emph{abstract probability spaces} $X = (X_\AbsMes, \mu_X)$, with morphisms being $\AbsMes$-morphisms $T_\AbsMes \colon X_\AbsMes \to Y_\AbsMes$ that are probability-preserving, and composition given by the $\AbsMes$-composition law.
\item[(vi)]  \cite[Definition 1.1(i)]{jt-foundational} $\CH$ is the category of compact Hausdorff spaces, with morphisms being continuous maps $T \colon X \to Y$.
\item[(vii)]  \cite[Definition 5.1(i)]{jt-foundational} $\CHProb$ is the category of compact Hausdorff probability spaces $X = (X_\CH, \mu_X)$, where the underlying $\SigmaAlg$-algebra for $\mu_X$ is the Baire algebra, and the morphisms given by continuous probability-preserving maps $T \colon X \to Y$.
\end{itemize}
Unless otherwise specified, composition of morphisms is given by the $\Set$-composition law.
\end{definition}

As discussed in \cite{jt-foundational}, the categories $\Set$, $\CH$, $\ConcMes$, $\AbsMes$ admit arbitrary categorical products (and dually, $\SigmaAlg$ admits arbitrary categorical coproducts), while $\ProbAlg$, $\ConcProb$, $\CHProb$ admit arbitrary tensor products\footnote{See \cite[Section A.3]{jt-foundational} for an introduction and comparison of the notions of categorical (co)products and tensor products based on the concept of (semicartesian) symmetric monoidal categories.}.

The faithful forgetful functors displayed as unlabeled blue arrows in Figure \ref{fig:categories} are defined in the obvious fashion.  Several additional functors in this diagram were defined in \cite{jt-foundational}, and we give an abridged description here (describing the action of the functors on objects but not morphisms):

\begin{definition}[Basic functors]\label{basic-func}\ 
\begin{itemize}
  \item[(i)]  \cite[Definition 6.1(viii)]{jt-foundational}  If $X = (X_\Set, \Sigma_X)$ is a $\ConcMes$-space, $\Abs(X) = X_\AbsMes = \Sigma_X$ is its abstraction.  Similarly, if $X = (X_\ConcMes, \mu_X)$ is a $\ConcProb$-space, $\Abs(X) = (\Abs(X_\ConcMes), \mu_X)$ is its abstraction.
  \item[(ii)]  \cite[Definition 6.1(ix)]{jt-foundational} If $X = (X_\AbsMes, \mu_X)$ is an $\AbsProb$-space, $\Mes(X) = X_\ProbAlg$ is the associated $\ProbAlg$-space formed by quotienting out the null ideal of $X_\SigmaAlg$.
  \item[(iii)]  \cite[Definition 3.1(v)]{jt-foundational}  If $X$ is a $\CH$-space, $\BaireFunc(X) = X_\ConcMes$ is the $\ConcMes$-space formed by endowing $X_\Set$ with the Baire $\SigmaAlg$-algebra $\Sigma_{X_\Set}=\Baire(X)$.
  \item[(iv)] If $X = (\Inc(X)_\SigmaAlg, \mu_X)$ is a $\ProbAlg$-space, then $\Inc(X) = \Inc(X)_\SigmaAlg$ is the associated $\AbsProb$-space.
\end{itemize}
\end{definition}

All the functors in blue in Figure \ref{fig:categories} are declared to be casting functors in the sense of Definition \ref{cast}; for instance, we have a casting functor $\Cast_{\CHProb \to \ProbAlg}$ defined by
$$\Cast_{\CHProb \to \ProbAlg} = \Mes \circ \Abs \circ \Cast_{\CHProb \to \ConcProb}.$$
In \cite[Lemma 6.4(ii)]{jt-foundational} it was observed that there is a natural monomorphism $\iota$ from $\Inc \circ \Mes$ to $\ident_\AbsProb$, thus we have an $\AbsProb$-monomorphism $\iota \colon \Inc(X_\ProbAlg) \to X$ naturally assigned to each $\AbsProb$-space $X$.

We will make frequent use of the Riesz representation theorem for compact Hausdorff spaces (equipped with the Baire $\sigma$-algebra).  Given a $\CH$-space $X$, we let $C(X)$ be the $*$-algebra of continuous functions from $X$ to $\C$.

\begin{theorem}[Riesz representation theorem in $\CH$]\label{rrt}  Let $X$ be a $\CH$-space.  Then there is a one-to-one correspondence between Baire probability measures $\mu_X$ on $X$ and linear functionals $\lambda \colon C(X) \to \C$ which are non-negative ($\lambda(f) \geq 0$ whenever $f \geq 0$) and are such that $\lambda(1)=1$, with each measure $\mu_X$ being associated to the integration functional
$$ \lambda_\mu(f) := \int_X f\ d\mu.$$
Furthermore, the Baire probability measures are automatically Baire--Radon (one has $\mu(E) = \sup \{ \mu(F): F \subset E, F \hbox{ is } G_\delta \}$ for all Baire sets $E$).
\end{theorem}

\begin{proof}  See for instance \cite[\S 2]{varadarajan-riesz}, \cite[Theorem 3.3]{sunder}, \cite{hartig}, or \cite{garling} for the first claim, and \cite[Proposition 4.2]{jt-foundational} for the second claim.  See also \cite[Theorem 5.4]{jt-foundational} for further variations of the Riesz representation theorem for locally compact Hausdorff spaces.
\end{proof}

\subsection{The canonical model}

A more non-trivial functor that we will need in our arguments is the \emph{canonical model functor}:

\begin{theorem}[Canonical model functor]\label{canon}  There exists a faithful and full functor $\Stone \colon \ProbAlg \to \CHProb$ with the following properties:
\begin{itemize}
\item[(i)] (Concrete model)  There is a natural isomorphism between $\Cast_{\CHProb \to \ProbAlg} \circ \Stone$ and the identity functor $\ident_{\ProbAlg}$.
\item[(ii)]  (Inclusion) There is a natural monomorphism $\iota'$ from $\Inc$ to $\Cast_{\CHProb \to \AbsProb} \circ \Stone$. (See Figure \ref{fig:canon-diag}.)
\item[(iii)]  (Strong Lusin property)  For any $\ProbAlg$-space $X$, we have the equivalence $L^\infty(\Stone(X)) \equiv C(\Stone(X))$.  That is to say, every bounded measurable function on $\Stone(X)$ agrees almost everywhere with precisely one continuous function on $\Stone(X)$.
\item[(iv)]  (Canonical representation) 
 If $X$ is a $\ProbAlg$-space and $K$ is a $\CH$-space, then to every $\AbsMes$-morphism $f \colon \Inc(X) \to K$ there is a unique $\CH$-morphism $\tilde f \colon \Stone(X) \to K$ which represents (or "extends") $f$ in the sense that $f =_\AbsMes \tilde f \circ \iota'$, where $\iota' \colon \Inc(X) \to \Stone(X)$ is the canonical $\AbsProb$-morphism.  In other words, one has an equivalence
$$ \Hom_\AbsMes(\Inc(X) \to K) \equiv \Hom_\CH(\Stone(X) \to K).$$
  \item[(v)] (Surjective morphisms)  If $T\colon X \to Y$ is a $\ProbAlg$-morphism, then $\Stone(T) \colon \Stone(X) \to \Stone(Y)$ is surjective\footnote{This is consistent with the fact that all morphisms in $\ProbAlg$ are epimorphisms; see \cite[Lemma 6.3(iii)]{jt-foundational}.}.
\end{itemize}
\end{theorem}

\begin{proof}  Two constructions of $\Stone$ are given in \cite{jt-foundational}.  In the first construction, $\Stone(X)$ is the Gelfand spectrum of the space $L^\infty(X)$ of bounded abstractly measurable functions on a $\ProbAlg$-space $X$ (viewed as a commutative $C^*$-algebra), with the probability measure provided by the Riesz representation theorem (Theorem \ref{rrt}), and the relevant claims (i)-(v) are verified in \cite[Theorem 7.4, Proposition 7.5, Proposition 7.8, Proposition 7.9]{jt-foundational}; see also \cite{segal}, \cite{DNP}, \cite{ellis}, \cite[\S 12.3, 13.4]{EFHN} for very similar constructions and results.  In the second construction, $\Stone(X)$ is the Stone space associated via Stone duality to $\Inc(X)_\SigmaAlg$ (viewed as a Boolean algebra), and the measure of a Baire set in $\Stone(X)$ is equal to the measure of the unique element of $X$ that generates a clopen subset of $\Stone(X)$ that differs from that Baire set by a Baire-meager set; see \cite[Theorem 9.10]{jt-foundational}.  Some closely related constructions also appear in \cite{fremlinvol3}, \cite{doob-ratio} (see \cite[Remarks 9.9, 9.12 and 9.13]{jt-foundational} for a comparison). However, for the purposes of this paper, the construction of $\Stone$ can be taken as a "black box".
\end{proof}

 \begin{figure}
    \centering
    \begin{tikzcd}
    \Stone(X) \arrow[r,"\Stone(\pi)"] & \Stone(Y) \\
    \Inc(X) \arrow[u, "\iota'", hook] \arrow[r, "\Inc(\pi)"] & \Inc(Y) \arrow[u, "\iota'", hook]
    \end{tikzcd}
    \caption{Every $\ProbAlg$-morphism $\pi: X \to Y$ gives rise to an $\AbsProb$-morphism $\Inc(\pi) \colon \Inc(X) \to \Inc(Y)$ and a $\CHProb$-morphism $\Stone(\pi) \colon \Stone(X) \to \Stone(Y)$, linked by the above commutative diagram in $\AbsProb$, with $\iota'$ the canonical inclusions. Casting functors have been suppressed to reduce clutter.}
    \label{fig:canon-diag}
\end{figure} 

As one quick application of the canonical model, one can immediately define\footnote{This may appear to be potentially circular because, as noted above, one way to construct $\Stone$ is to first define an abstract $L^\infty$ space $L^\infty(Y)$ and then evaluate the Gelfand spectrum.  However, it is possible to construct $\Stone$ without using the notion of abstract $L^\infty$, and in any event it is not difficult to see using \eqref{linf} below that the definition of abstract $L^\infty$ given here agrees up to natural isomorphism with the one in \cite[\S 7]{jt-foundational}.} the abstract $L^p$-space $L^p(Y)$ of a $\ProbAlg$-space $Y$ for any $1 \leq p \leq \infty$ simply by declaring
$$ L^p(Y) := L^p(\Stone(Y));$$
in this paper we will only need this construction for $p=2,\infty$.
Furthermore one can define an abstract integral $\int_Y \colon L^p(Y) \to \C$ on these spaces just by using the concrete integral $\int_{\Stone(Y)} \colon L^p(\Stone(Y)) \to \C$.  Using the fact that $Y$ and $\Stone(Y)$ are isomorphic in $\ProbAlg$, we see that every $E \in \Inc(Y)_\SigmaAlg$ gives rise to an indicator function $1_E \in L^\infty(Y) \subset L^2(Y)$ which obeys the expected properties, e.g., $1_E 1_F = 1_{E \wedge F}$ and $\int_Y 1_E = \mu_Y(E)$.  We remark that an equivalent construction of $L^p$-spaces on general measure algebras was given by Fremlin \cite{fremlinvol3}; see \cite[Remark 9.13]{jt-foundational} for a detailed comparison.  Also, if $Y$ is a $\ConcProb$-space, then the concrete $L^p$-spaces $L^2(Y), L^\infty(Y)$ can be naturally identified with their abstract counterparts $L^2(Y_\ProbAlg), L^\infty(Y_\ProbAlg)$.  For instance, $L^2(Y)$ and $L^2(Y_\ProbAlg)$ can both be identified with the Hilbert spaces generated by formal indicator functions $1_E, E \in Y_\SigmaAlg$ with inner product $\langle 1_E, 1_F \rangle = \mu_Y(E \wedge F)$, and $L^\infty(Y)$ and $L^\infty(Y_\ProbAlg)$ can both be identified with the operator norm linear combinations of indicators $1_E, E \in Y_\SigmaAlg$, viewed as bounded linear operators on $L^2(Y)$.  As such we will freely identify the two $L^p$ constructions, so that
$$ L^p(Y) \equiv L^p(Y_\ProbAlg)$$
for $Y \in \ConcProb$ and $p=2,\infty$.  Thus for instance we have the equivalences
$$ L^\infty(Y) \equiv L^\infty(Y_\ProbAlg) = L^\infty(\Stone(Y)) \equiv C(\Stone(Y)).$$

If $\pi \colon X \to Y$ is a $\ProbAlg$-morphism, we define the pullback maps $\pi^* \colon L^\infty(Y) \to L^\infty(X)$ and $\pi^* \colon L^2(Y) \to L^2(X)$ by 
$$ \pi^* f := f \circ \Stone(\pi).$$
Note that with our identifications this is compatible with the pullback maps $\pi^* \colon L^\infty(Y) \to L^\infty(X)$ and $\pi^* \colon L^2(Y) \to L^2(X)$ associated to a $\ConcProb$-morphism $\pi$ by the Koopman operator
$$ \pi^* f := f \circ \pi.$$
Thus by abuse of notation we shall use the same notation $\pi^*$ for both maps.

\subsection{Dynamical categories}

Let $\Gamma$ be a group and $K$ a compact Hausdorff group (that is to say, $K$ is a $\CH$-space with a group structure such that the group operations $()^{-1} \colon K \to K$, $\cdot \colon K \times^\CH K \to K$ are $\CH$-morphisms).  We write $K^\op$ for the opposite group of $K$, that is to say the collection of formal objects $k^\op, k \in K$ with group law
$$ k_1^\op \cdot k_2^\op := (k_2 \cdot k_1)^{\op}$$
and inverse law
$$ (k^\op)^{-1} := (k^{-1}).$$
We then form the product group $\Gamma \times K^\op$, which contains $\Gamma$ as a normal subgroup.  Using the general construction in Definition \ref{dynamic}, one can define dynamical analogues $\Cat_\Gamma$, $\Cat_{\Gamma \times K^\op}$ of the categories $\Cat = \ProbAlg$, $\AbsProb$, $\ConcProb$, $\CHProb$, $\CH$.  All the functors in Figure \ref{fig:categories} between these five categories then extend to their dynamical counterparts as indicated in that figure, and there are also forgetful functors from $\Cat_{\Gamma \times K^\op}$ to $\Cat_\Gamma$ and from $\Cat_\Gamma$ to $\Cat$, defined in the obvious fashion.

\begin{remark}  As a gross oversimplification,  topological dynamics is the study of the category $\CH_\Gamma$, concrete ergodic theory is the study of $\ConcProb_\Gamma$, and the ergodic theory of measure algebras is the study of $\ProbAlgG$. Finally, the theory of topological models of ergodic theory systems is the study of $\CHProb_\Gamma$ (and its relationship with the other categories just mentioned).
\end{remark}

\begin{remark} We will also make occasional use of the categories $\Cat_{K^\op}$ for $\Cat$ as above, with forgetful functors from $\Cat_{\Gamma \times K^\op}$ to $\Cat_{K^\op}$ and from $\Cat_{K^\op}$ to $\Cat$.  We have not placed these categories and functors into Figure \ref{fig:categories} in order to reduce clutter.
\end{remark}

We have now defined all the functors marked in blue in Figure \ref{fig:categories}; we deem these all to be casting functors in the sense of Definition \ref{cast}.  It is a routine matter to verify that these functors all commute with each other.

Next, we introduce an invariant factor functor $\Inv$ to define on both $\ProbAlgG$ and $\ProbAlgGK$:

\begin{definition}[Invariant factor functor]\label{inv-def}  Let $\Gamma'$ be equal to either $\Gamma$ or $\Gamma \times K^\op$ (so in particular $\Gamma$ be a normal subgroup of $\Gamma'$).  
\begin{itemize}
    \item[(i)]  If $X = (X_\ProbAlg,T_X)$ is a $\ProbAlg_{\Gamma'}$-system, we define $\Inv(X)$ to be the $\ProbAlg_{\Gamma'}$-system with $\SigmaAlg$-algebra
    $$ \Inc(\Inv(X))_\SigmaAlg := \{ E \in \Inc(X)_\SigmaAlg: (T_X^\gamma)^*(E) = E \,\forall \gamma \in \Gamma \}$$
    and measure
    $$ \mu_{\Inv(X)}(E) = \mu_X(E)$$
    for all $E \in \Inc(\Inv(X))_\SigmaAlg$, and action defined by
    $$ (T_{\Inv(X)}^\gamma)^*(E) := (T_X^\gamma)^*(E)$$
    for all $E \in \Inc(X)_\SigmaAlg$ and $\gamma \in \Gamma'$.  In particular, if $\Gamma'=\Gamma$ then $T_{\Inv(X)}$ is now just the identity action.  
    \item[(ii)]  If $f \colon X \to Y$ is a $\ProbAlg_{\Gamma'}$-morphism, we define $\Inv(f) \colon \Inv(X) \to \Inv(Y)$ by defining 
    $$\Inv(f)^*(E) = f^*(E)$$
    whenever $E \in \Inc(Y)_\SigmaAlg$.
    \item[(iii)] A $\ProbAlgG$-system $X$ is said to be \emph{ergodic} if $\Inv(X)$ is trivial, in the sense that $\Inc(\Inv(X))_\SigmaAlg = \{0,1\}$,  If $\Cat = \AbsProb$, $\ConcProb$, $\CHProb$, a $\Cat_\Gamma$-system is said to be ergodic if its cast to $\ProbAlgG$ is ergodic. 
\end{itemize}
\end{definition}

Thus, for instance, if $X$ is a $\ConcProb_\Gamma$-system, $X$ is ergodic if the only measurable subsets $E$ of $X$ that are essentially invariant in the sense that $\mu_X(T^\gamma E \Delta E ) =0$ for all $\gamma \in \Gamma$ have measure $0$ or $1$.

It is not difficult to verify that $\Inv$ is a functor on both $\ProbAlgG$ and $\ProbAlg_{\Gamma \times K^{op}}$.  There is a natural projection from $\ident_{\ProbAlg_{\Gamma}}$ to $\Inv$ that gives a $\ProbAlgG$-morphism $\pi \colon X \to \Inv(X)$ from any $\ProbAlgG$-system $X$ to its invariant factor $\Inv(X)$, defined by setting $\Inc(\pi)_\SigmaAlg \colon \Inc(\Inv(X))_\SigmaAlg \to \Inc(X)_\SigmaAlg$ to be the inclusion map.  Using this morphism, one can view $L^2(\Inv(X)) = L^2(\Stone(\Inv(X)))$ as a subspace of $L^2(X)= L^2(\Stone(X))$ (identifying any $f \in L^2(\Stone(\Inv(X)))$ with its counterpart $f \circ \Stone(\pi)$ in $L^2(\Stone(X))$).  Meanwhile, each shift $T^\gamma \colon X \to X$ induces a unitary Koopman operator\footnote{One can interpret $L^2$ here as a functor from $\ProbAlg^\op$ to the category $\Hilb$ of Hilbert spaces (in which the morphisms are isometries); applying Hilbert space duality, one can also form a functor $*L^2$ from $\ProbAlg$ to the category $\OpHilb$ of dual Hilbert spaces (in which the morphisms are co-isometries).  Much of the discussion here can be "factored through $*L^2$", reflecting the fact that the mean ergodic theorem can be formulated in terms of Hilbert space dynamics instead of measurable dynamics, but we will not need to adopt this perspective here.} $(T^\gamma)^* \colon L^2(X) \to L^2(X)$.

We now record the following basic relationship between the invariant factor and any other factor of a system.  Given two closed subspaces $H_2 \subset H_1 \subset H$ of a Hilbert space $H$, we use $H_1 \ominus H_2$ to denote the orthogonal complement of $H_2$ in $H_1$.

\begin{proposition}[Relative orthogonality]\label{rel}  Let $\pi \colon X \to Y$ be a $\ProbAlgG$-morphism.  We identify $L^2(Y)$ with a closed subspace of $L^2(X)$, $L^2(\Inv(Y))$ with a closed subspace of $L^2(Y)$, and $L^2(\Inv(X))$ with a closed subspace of $L^2(X)$.  Then 
$L^2(Y) \ominus L^2(\Inv(Y))$ is orthogonal to $L^2(\Inv(X)) \ominus L^2(\Inv(Y))$ (viewing both spaces as subspaces of $L^2(X)$).
\end{proposition}

\begin{proof}  If $f \in L^2(\Inv(X))$, then $f$ is invariant, and thus its orthogonal projection $\pi(f)$ to $L^2(Y)$ is also invariant, and thus lies in $L^2(\Inv(Y))$.  Thus $L^2(\Inv(X))$ is orthogonal to $L^2(Y) \ominus L^2(\Inv(Y))$, giving the claim\footnote{We thank the anonymous referee for this simple argument.}.
\end{proof}

We also record a simple computation of an invariant factor:

\begin{lemma}[Invariant factor of translation action]\label{trans}  Let $X$ be a $\ConcProb$-space and $K$ a compact Hausdorff group, which we view as a $\ConcProb$-space endowed with (Baire) Haar measure.  Suppose that an equivalence class $[f] \in L^2(X \times^\ConcProb K)$ of a square-integrable $\ConcMes$-morphism $f \colon X \times^\ConcMes K \to \C$ is $K^\op$-invariant in the sense that $[k_0^{\op} f] = [f]$ in $L^2(X \times^\ConcProb K)$ for all $k_0^\op \in K^\op$, where
$$ k_0^{\op} f(x, k) := f(x, kk_0^{-1}).$$
Then $[f]$ arises from an element of $L^2(X)$ (viewed as a subspace of $L^2(X \times^\ConcProb K)$ by applying the functor $L^2$ to the projection map).
\end{lemma}

\begin{proof}  The potential subtlety here is that each $k_0^{\op} f$ may differ from $f$ on a null set that depends on $k_0^\op$, and that $k_0^\op$ can range over uncountably many values.  Fortunately these issues can be avoided by using duality and applying Fubini's theorem.  Indeed, for any $[g] \in L^2(K)\subset L^2(X \times^\ConcProb K)$, we have
\begin{align*}
     &\langle [f], [g] \rangle_{L^2(X \times^\ConcProb K)} \\
     &= \int_K \langle [k_0^\op f], [g] \rangle_{L^2(X \times^\ConcProb K)}\ d\Haar_K(k_0) \\
     &= \int_K \int_X \int_K f(x,kk_0^{-1}) g(k)\ d\Haar_K(k) d\mu_X(x) d\Haar_K(k_0) \\
     &= \int_X \int_K \left(\int_K f(x,k')\ d\Haar_K(k')\right) g(k)\ d\Haar_K(k) d\mu_X(x) 
\end{align*}
and thus $f$ is equal in $L^2(X \times^\ConcProb K)$ to the function $x \mapsto \int_K f(x,k')\ d\Haar_K(k')$, which lies in $L^2(X)$, and the claim follows.
\end{proof}

\subsection{Conditional elements}

The final pair of functors we need to define in Figure \ref{fig:categories} are the conditional functors $\Cond_Y \colon \ConcMes \to \Set$, $\CondTilde_Y \colon \CH \to \Set$ defined for any $\ProbAlg$-space $Y$.  

\begin{definition}[Conditional functors]\label{cond-def}  Let $Y$ be a $\ProbAlg$-space (or an object that can be casted to $\ProbAlg$, such as a $\ProbAlgG$-space).  
\begin{itemize}
\item[(i)] \cite{jt19} If $K$ is a $\ConcMes$-space, we define $$\Cond_Y(K) := \Hom_\AbsMes(\Inc(Y) \to K)$$ to be the $\Set$-space of all $\AbsMes$-morphisms from $\Inc(Y)_\AbsMes$ to $K_\AbsMes$; elements of $\Cond_Y(K)$ will be referred to as \emph{conditional elements} of $K$ (in contrast to the actual elements of $K$, which we now call \emph{classical elements}).  If $f \colon K \to L$ is a $\ConcMes$-morphism, we define $\Cond_Y(f) \colon \Cond_Y(K) \to \Cond_Y(L)$ to be the $\Set$-morphism $\Cond(f)(k) := f \circ k$ for all $k \in \Cond_Y(K)$, where we use the composition law in $\AbsMes$.
\item[(ii)] If $K$ is a $\CH$-space, we define $ \CondTilde_Y(K) := \Hom_\CH(\Stone(Y) \to K)$ to be the $\Set$-space of all $\CH$-morphisms from $\Stone(Y)$ to $K$.  If $f \colon K \to L$ is a $\CH$-morphism, we define $\CondTilde_Y(f) \colon \CondTilde_Y(K) \to \CondTilde_Y(L)$ to be the $\Set$-morphism $\Cond(f)(k) := f \circ k$ for all $k \in \CondTilde_{Y}(K)$.
\end{itemize}
\end{definition}

It is easy to see that $\Cond_Y \colon \ConcMes \to \Set$ is a functor.  Every classical element $k \in K$ gives rise to a conditional element $k \in \Cond_Y(K)$, defined via the pullback operation $k_\SigmaAlg$ by
$$ k_\SigmaAlg(E) = 1_{k \in E}$$
for all $E \in K_\SigmaAlg=\Baire(K)$; this gives a natural transformation from the forgetful functor $\Forget_{\ConcMes \to \Set}$ to $\Cond_Y$.  The significance of the spaces $\Cond_Y(K)$ for us is that they will be used to describe the components $\rho_\gamma$ of an abstract $K$-valued cocycle $\rho$ on a $\ProbAlgG$-space $Y$ when $K$ is a compact Hausdorff group.

From Theorem \ref{canon}(iv) we have the fundamental equivalence
\begin{equation}\label{rep}
\Cond_Y(K) \equiv \CondTilde_Y(K)
\end{equation}
for any $\CH$-space $K$; in fact a routine verification shows that we in fact have a natural isomorphism between $\CondTilde_Y$ and $\Cond_Y \circ \BaireFunc$.  Thus, every conditional element $k$ of $K$ has a canonical representation as a $\CH$-morphism $\tilde k \colon \Stone(Y) \to K$.  This representation plays a role analogous to "representation theorems" (also known as liftings, cf. \cite[Chapter  34]{fremlinvol3}), such as the one of Maharam \cite{maharam}, that represent various abstract measurable maps by concrete ones, see \cite[Corollary 5.8]{jt19}.  However, there are several advantages to the representation \eqref{rep}.  Firstly, the concrete representation $\tilde k$ is automatically continuous.  Secondly, $Y$ is not required to be modeled by a space with a complete $\sigma$-algebra.  Thirdly, the representation is canonical and functorial, in contrast to the Maharam theorem which does not assert any uniqueness of the representation and provides no functoriality properties.

As one application of \eqref{rep}, we see that for a $\ProbAlg$-space $Y$, we have
\begin{equation}\label{linf}
 L^\infty(Y) = L^\infty(\Stone(Y)) \equiv C(\Stone(Y)) = \bigcup_{R>0} \CondTilde_Y(B(0,R)) \equiv \bigcup_{R>0} \Cond_Y(B(0,R))
 \end{equation}
 where $B(0,R)$ is the closed unit ball of radius $R$ in $\C$ (and we view $\CondTilde_Y(B(0,R))$, $R>0$ as a nested increasing collection of sets, and similarly for $\Cond_Y(B(0,R))$).  Thus we can identify $L^\infty(Y)$ with the bounded conditional elements of $\C$. To put it another way, every bounded $\AbsMes$-morphism from $\Inc(Y)$ to $\C$ has a unique continuous extension to $\Stone(Y)$.

We present a useful technical lemma:

\begin{lemma}[Base change is a natural monomorphism]\label{base-change}  Let $\pi \colon X \to Y$ be a $\ProbAlg$-morphism.  Then the base change natural transformation $\pi^* \colon \CondTilde_Y \to \CondTilde_X$ that assigns to each $\CH$-space $K$ the $\Set$-morphism $\pi^* \colon \CondTilde_Y(K) \to \CondTilde_X(K)$ defined by $\pi^* k := k \circ \Stone(\pi)$ is a natural monomorphism.  In particular, the abstract pullback map $\pi^* \colon \Cond_Y(K) \to \Cond_X(K)$ is injective for any $\CH$-space $K$.
\end{lemma}

\begin{proof}  This follows from the surjectivity of $\Stone(\pi)$ (Theorem \ref{canon}(iv)).
\end{proof}

Another important property of the functor $\Cond_Y$ (or $\CondTilde_Y$) is that it preserves surjectivity:

\begin{lemma}[Surjectivity preservation]\label{surj}  Let $Y$ be a $\ProbAlg$-space.
\begin{itemize}
\item[(i)]  The $\SigmaAlg$-algebra $\Inc(Y)_\SigmaAlg$ is complete (that is, every collection of sets $(E_\alpha)_{\alpha \in A}$ in $\Inc(Y)_\SigmaAlg$ has a supremum $\bigvee_{\alpha \in A} E_\alpha$).
\item[(ii)]  The canonical model $\Stone(Y)$ is extremally disconnected (that is, the closure of every open subset of $\Stone(Y)$ is also open).
\item[(iii)]  If $f \colon Z \to W$ is a surjective $\CH$-morphism, then the $\Set$-morphism $\tilde f \colon \CondTilde_Y(Z) \to \CondTilde_Y(W)$ defined by $\tilde f(z) := f \circ z$ is also surjective.  By \eqref{rep}, we obtain a similar claim for $\Cond_Y$.
\end{itemize}
\end{lemma}

\begin{proof}  The claim (i) is well known: if $(E_\alpha)_{\alpha \in A}$ is a family in $\Inc(Y)_\SigmaAlg$, we let ${\mathcal F}$ be the collection of all finite joins of the $E_\alpha$, and we then extract a sequence $F_n$ in ${\mathcal F}$ whose measures $\mu_Y(F_n)$ converges to $\sup_{F \in {\mathcal F}} \mu_Y(F)$.  If one then sets $F := \bigvee_n F_n$, one can readily check that $\mu_Y(F \backslash E_\alpha) = 0$ for all $\alpha \in A$, and hence (as $Y$ is a $\ProbAlg$-space) $E_\alpha$ is contained in $F$, so that $F$ is an upper bound for the $E_\alpha$.  Conversely, any upper bound for the $E_\alpha$ has to have measure at least $\sup_{F \in {\mathcal F}} \mu_Y(F)$; by another appeal to the fact that $Y$ is a $\ProbAlg$-space, we conclude that $F$ is a least upper bound to the $E_\alpha$, giving the claim.

The claim (ii) follows from (i) and the well-known fact that the Stone dual of a complete Boolean algebra is extremally disconnected.  Indeed, any open subset of $\Stone(Y)$ is the union of clopen sets, each of which is associated via Stone duality to an element $E_\alpha, \alpha \in A$ of $\Inc(Y)_\SigmaAlg$.  If we take $F := \bigvee_{\alpha \in A} E_\alpha$ to be the supremum of the $E_\alpha$, it is routine to verify that the clopen set associated to $F$ is the closure of the union of the clopen sets associated to $E_\alpha$ giving the claim.

From (ii) and a theorem of Gleason \cite[Theorem 2.5]{gleason}, we conclude that $\Stone(Y)$ is projective in the category $\CH$; that is to say, given a surjective $\CH$-morphism $f \colon Z \to W$, any $\CH$-morphism $w \colon Y \to W$ has a lift $z \colon Y \to Z$ such that $f \circ z = w$.  This gives (iii).
\end{proof}

\section{Lifting conditional elements from group quotients}\label{lift-sec}

Let $Y$ be a $\ProbAlg$-space, and let $K$ be a compact Hausdorff group.  Then $\CondTilde_Y(K) = \Hom_\CH(\Stone(Y) \to K)$ obviously has the structure of a group (and hence, by \eqref{rep}, so does $\Cond_Y(K)$). If $H$ is a closed subgroup of $K$ (not necessarily normal), then $\CondTilde_Y(H)$ can be identified with a subgroup of $\CondTilde_Y(K)$ in the obvious fashion, so we may form the quotient spaces
$$ \CondTilde_Y(K) /  \CondTilde_Y(H), 
\CondTilde_Y(H) \backslash  \CondTilde_Y(K).$$
Meanwhile, the quotient spaces $K/H, H \backslash K$ have the structure of $\CH$-spaces, so we may also form the spaces
$$ \CondTilde_Y(K/H), \CondTilde_Y(H \backslash K).$$
The projection map from $K$ to $K/H$ is a $\CH$-morphism, hence induces a corresponding map from $\CondTilde_Y(K)$ to $\CondTilde_Y(K/H)$.  Any two elements of the group $\CondTilde_Y(K)$ that differ on the right by an element of the subgroup $\CondTilde_Y(H)$ can easily be seen to map to the same element of $\CondTilde_Y(K/H)$.  Thus we have a canonical map from $ \CondTilde_Y(K) /  \CondTilde_Y(H)$ to $\CondTilde_Y(K/H)$, and similarly a canonical map from
 $\CondTilde_Y(H) \backslash \CondTilde_Y(K)$ to $\CondTilde_Y(H \backslash K)$.  It is easy to see that this canonical map is injective.  Surjectivity follows from Lemma \ref{surj}(iii), and we thus have the following innocuous-looking (but quite important) result:
 
\begin{theorem}[Lifting conditional elements]\label{lift}  With the notation as above, we have the identifications
\begin{align*}
  \CondTilde_Y(K) /  \CondTilde_Y(H)&= \CondTilde_Y(K/H),  \\
  \CondTilde_Y(H)  \backslash  \CondTilde_Y(K)&= \CondTilde_Y(H \backslash K).
\end{align*} 
In particular, by \eqref{rep}, we also have the identifications
\begin{align*}
  \Cond_Y(K) /  \Cond_Y(H) &= \Cond_Y(K/H),  \\
 \Cond_Y(H)  \backslash  \Cond_Y(K)&= \Cond_Y(H \backslash K).
\end{align*}
\end{theorem}

\section{Extensions and skew-products in probability algebras}\label{product-sec}

In Definition \ref{concrete-ext} we defined the notion of a cocycle, skew-product, and extension in the  $\ConcProb_\Gamma$-category.  We can similarly define such concepts in the $\CHProb_\Gamma$-category:

\begin{definition}[$\CHProb$ skew-products and extensions]\label{chprob-ext}  We define the notions of a $\CHProb_\Gamma$-cocycle, $\CHProb_\Gamma$-homogeneous skew-product, $\CHProb_\Gamma$-group skew-product, $\CHProb_\Gamma$-homogeneous extension, and $\CHProb_\Gamma$-group extension exactly as in Definition \ref{concrete-ext}, replacing all occurrences of $\ConcProb$, $\ConcProb_\Gamma$ with $\CHProb$, $\CHProb_\Gamma$ respectively, and requiring the cocycle $\rho_\gamma$ (and also the coordinate map $\theta$) to be a $\CH$-morphism (i.e., continuous and measurable) rather than merely being a $\ConcMes$-morphism (i.e., measurable).  
\end{definition}

We can leverage this definition using the canonical model  functor $\Stone$ to define the analogous notions in $\ProbAlg$:

\begin{definition}[$\ProbAlgG$ skew-products and extensions]\label{abstract-ext}  Let $\Gamma$ be a group, Let $Y = (Y_\ProbAlg, T_Y)$ be a $\ProbAlgG$-system,  let $K$ be a compact Hausdorff group, and let $L$ be a closed subgroup of $K$.
\begin{itemize}
    \item[(i)]  A \emph{$K$-valued $\ProbAlgG$-cocycle} on $Y$ is a tuple $\rho = (\rho_\gamma)_{\gamma \in \Gamma}$ of conditional elements $\rho_\gamma \in \Cond_Y(K)$ of $K$ that obeys the \emph{cocycle equation}
    \begin{equation}\label{cocycle} 
    \rho_{\gamma  \gamma'} = (\rho_\gamma \circ \Inc(T_Y^{\gamma'})) \rho_{\gamma'}
    \end{equation}
    for all $\gamma,\gamma' \in \Gamma$, using the group structure on $\Cond_Y(K)$.  We define the \emph{canonical representation} $\tilde \rho = (\tilde \rho_\gamma)_{\gamma \in \Gamma}$ of this cocycle to be the elements $\tilde \rho_\gamma \in \CondTilde(Y)$ associated to $\rho_\gamma$ by the natural isomorphism in \eqref{rep}.  Note that $\tilde \rho$ is a $K$-valued $\CHProb_\Gamma$-cocycle on $\Stone(Y)$.
    \item[(ii)] If $\rho$ is a $K$-valued $\ProbAlgG$-cocycle, we define the \emph{$\ProbAlgG$-group skew-product} $Y \rtimes^{\ProbAlgG}_\rho K$ and  \emph{$\ProbAlgG$-homogeneous skew-product} $Y \rtimes^{\ProbAlgG}_\rho K/L$ by the formulae
    $$ Y \rtimes^{\ProbAlgG}_\rho K := ( \Stone(Y) \rtimes^{\CHProb_\Gamma}_{\tilde \rho} K )_{\ProbAlgG}$$
    and 
    $$ Y \rtimes^{\ProbAlgG}_\rho K/L := ( \Stone(Y) \rtimes^{\CHProb_\Gamma}_{\tilde \rho} K/L )_{\ProbAlgG}.$$
    \item[(iii)]  An \emph{$\ProbAlgG$-homogenous extension} of $Y$ by $K/L$ is a tuple $(X, \pi, \theta, \rho)$, where $(X, \pi)$ is a $\ProbAlgG$-extension of $Y$, the \emph{vertical coordinate} $\theta \in \Cond_X(K/L)$ is such that $\Inc(\pi), \theta$ jointly generate $\Inc(X)_\SigmaAlg$, and $\rho = (\rho_\gamma)_{\gamma \in \Gamma}$ is a $K$-valued $\ProbAlgG$-cocycle such that
    \begin{equation}\label{theta-gam-abstract}
    \theta \circ \Inc(T_X^\gamma) = (\rho_\gamma \circ \Inc(\pi)) \theta
    \end{equation}
    for all $\gamma \in \Gamma$, using the action of $\Cond_X(K)$ on $\Cond_X(K/L)$ arising from Theorem \ref{lift}.  If $L$ is trivial, we refer to such a tuple as a \emph{$\ProbAlgG$-group extension} of $Y$ by $K$.
\end{itemize}
\end{definition}

Note that every $\ProbAlgG$-group skew-product $X = Y \rtimes^{\ProbAlgG}_\rho K$ is also a $\ProbAlgG$-homogeneous skew-extension, with factor map
$$ \pi :=  \Cast_{\CHProb_\Gamma \to \ProbAlgG}( {\pi}_{\Stone(Y)} )$$
and vertical coordinate
$$ \theta_X := \iota \circ \Inc( {\pi}_K ),$$
where ${\pi}_{\Stone(Y)} \colon \Stone(Y) \times^{\CHProb} K \to \Stone(Y)$, ${\pi}_K \colon \Stone(Y) \times^{\CHProb} K \to K$ are the canonical coordinate $\CHProb$-morphisms,
and $\iota \colon \Inc(K) \to K$ is the canonical $\AbsMes$-inclusion.  

Next, observe that a $\CHProb_\Gamma$-group skew-product $X = Y \rtimes^{\CHProb_\Gamma}_\rho K$ can automatically be promoted to a $\CHProb_{\Gamma \times K^\op}$-system, with the action of $\Gamma \times K^{\op}$ given by
    $$ T_X^{(\gamma,k_0^\op)}(y,k) := (T_Y^\gamma(y), \rho_\gamma(y) kk_0)$$
for $\gamma \in \Gamma$, $k_0^\op \in K^\op$, $y \in Y_\Set$, and $k \in K$.
One easily checks (using the Fubini--Tonelli theorem and the bi-invariance of Haar measure $\Haar_K$) that this is indeed a promotion of $X$ to a $\CHProb_{\Gamma \times K^{\op}}$-system (note it is necessary to work with the reversed group $K^\op$ here in order to maintain the group action property).  As a consequence, any $\ProbAlgG$-group skew-product $X = Y \rtimes^{\ProbAlgG}_\rho K$ can similarly be promoted to a $\ProbAlgGK$-system.

We close this section with a criterion for determining when a homogeneous extension is equivalent to a skew-product.

\begin{proposition}[Criterion for skew-product]\label{skew-crit} Let $Y = (Y_\ProbAlg, T_Y)$ be a $\ProbAlgG$-system, and let $K$ be a compact Hausdorff group, and let $L$ be a closed subgroup of $K$.  We view $K/L$ as a $\CHProb$-space, equipped with (Baire--Radon) Haar probability measure $\Haar_{K/L}$. Let $(X, \pi, \theta, \rho)$ be a $\ProbAlgG$-homogeneous extension of $Y$ by $K/L$.  Suppose that one has
\begin{equation}\label{intx}
\int_X (\pi^* f) (g \circ \theta) = \left(\int_Y f\right) \left(\int_{K/L} g\right)
\end{equation}
for all $f \in L^\infty(Y)$ and $g \in C(K/L)$ (where we use \eqref{linf} to interpret $g \circ \theta \in \Cond_X(\C)$ as an element of $L^\infty(X)$).
Then $X$ is isomorphic as an extension of $Y$ in $\ProbAlgG$ (i.e., isomorphic in $(\ProbAlgG \downarrow Y)$) to $Y \rtimes^{\ProbAlgG}_\rho K/L$.
\end{proposition}

\begin{proof}  
Applying the canonical model functor $\Stone$ to the $\ProbAlgG$-extension $(X, \pi, \theta, \rho)$ of $Y$ by $K/L$, we obtain a $\CHProb_\Gamma$-extension
$(\Stone(X), \Stone(\pi), \tilde \theta, \tilde \rho)$ of $\Stone(Y)$ by $K/L$, where $\tilde \rho$ is the functorial concrete representation of $\rho$ as a $K$-valued $\CHProb_\Gamma$-cocycle on $\Stone(Y)$, and $\tilde \theta \in \CondTilde_X(K/L)$ is the continuous representative of $\theta \in \Cond_X(K/L)$.  Let $\tilde \phi \colon \Stone(X) \to \Stone(Y) \times^\CH K/L$ be the $\CH$-morphism defined by
\begin{equation}\label{phidef}
\tilde \phi(\tilde x) := (\Stone(\pi)(\tilde x), \tilde \theta(\tilde x))
\end{equation}
for $\tilde x \in \Stone(X)_\Set$, then from \eqref{theta-gam-abstract} (viewed in the concrete model) we see that
 $\tilde \phi$ may be promoted to a $\CH_\Gamma$-morphism of extensions from $(\Stone(X), \Stone(\pi))$ to $(\Stone(Y) \rtimes_{\tilde \rho} K/L, \pi_{\Stone(Y)})$, where $\pi_{\Stone(Y)} \colon \Stone(Y) \rtimes_{\tilde \rho} K/L \to \Stone(Y)$ is the canonical $\CHProb_\Gamma$-morphism.  Meanwhile, the hypothesis \eqref{intx} implies that
$$ \int_{\Stone(X)} (f \otimes g) \circ \tilde \phi = \int_{\Stone(Y) \times^{\CHProb} K/L} f \otimes g$$
for all $f \in C(\Stone(Y))$ and $g \in C(K/L)$, where $f \otimes g \colon \Stone(Y) \times K/L$ is the $\CH$-morphism $(\tilde y,kL) \mapsto f(\tilde y) g(kL)$.  From the Stone--Weierstra{\ss} theorem, finite linear combinations of tensor products $f \otimes g$ are dense in $C(\Stone(Y) \times K/L)$, so on taking limits and using the Riesz representation theorem (Theorem \ref{rrt}) we conclude
that $\tilde \phi$ can be promoted to a $\CHProb$-morphism, and can therefore be promoted further to a $\CHProb_\Gamma$-morphism of extensions from $(\Stone(X), \pi)$ to $(\Stone(Y) \rtimes_{\tilde \rho} K/L, \pi_{\Stone(Y)})$.

Casting the $\CHProb_\Gamma$-morphism $\tilde \phi$ to $\ProbAlgG$ and applying canonical isomorphisms, we obtain a $\ProbAlgG$-morphism $\phi$ of extensions from $(X, \pi)$ to $(Y \rtimes_\rho K/L, \pi_Y)$, where $\pi_Y \colon Y \rtimes_\rho K/L \to Y$ is the canonical $\ProbAlgG$-morphism. The only remaining thing to do is to show that the morphism $\phi$ is in fact an \emph{isomorphism} in $\ProbAlgG$.  By chasing the definitions, we see that it suffices to show that the pullback map
$$ \Inc(\phi)_\SigmaAlg \colon \Inc(Y \times_\rho K/L)_\SigmaAlg \to \Inc(X)_\SigmaAlg$$
is bijective.  Injectivity is clear from the measure-preserving nature of $\phi$ (which ensures that non-zero elements map to non-zero elements), so it suffices to prove surjectivity.  But from \eqref{phidef} we see that the range of $ \Inc(\phi)_\SigmaAlg $ contains the range of $\Inc(\pi)_\SigmaAlg \colon \Inc(Y)_\SigmaAlg \to \Inc(X)_\SigmaAlg$, as well as the range of $\theta_\SigmaAlg \colon (K/L)_\SigmaAlg \to \Inc(Y)_\SigmaAlg$.  Since these two maps generate $\Inc(X)_\SigmaAlg$ by hypothesis, we obtain the required surjectivity.
\end{proof}

\section{First step: passing to a product system}\label{product-system}

After all the extensive preliminaries and setup, we are now ready to begin the proof of Theorem \ref{mackey-uncountable}(i); we will later deduce (ii) as a consequence of (i) starting in Section \ref{sixth}.  The strategy is to gradually build the diagram in Figure \ref{fig:diagram} by following the path sketched out in Section \ref{proof-method}.  Henceforth $\Gamma$ is a group, $K$ is a compact Hausdorff group, $Y$ is an ergodic $\ProbAlgG$-system, and $(X,\pi_{X \to Y})$ is an ergodic $\ProbAlgG$-group extension of $Y$ by $K$. 

Definition \ref{abstract-ext} provides us a vertical coordinate $\theta \in \Cond_X(K)$ and a $K$-valued $\ProbAlgG$ cocycle $\rho = (\rho_\gamma)_{\gamma \in \Gamma}$; \eqref{rep} then gives us a concrete model $\tilde \theta \in \CondTilde_X(K)$ of $\theta$, as well as a $K$-valued $\ConcProb_\Gamma$-cocycle $\tilde \rho = (\tilde \rho_\gamma)_{\gamma \in \Gamma}$.  

By Definition \ref{abstract-ext} again, we can construct the $\ProbAlgGK$-systems $X \rtimes_1 K = X \rtimes^{\ProbAlgG}_1 K$ and $Y \rtimes_\rho K = Y \rtimes^{\ProbAlgG}_\rho K$, where $1$ denotes the trivial cocycle.  Observe that the map $\tilde \Pi \colon \Stone(X) \rtimes^{\CHProb_\Gamma}_1 K \to \Stone(Y) \rtimes^{\CHProb_\Gamma}_{\tilde \rho} K$ defined by
\begin{equation}\label{Pi-def}
\tilde \Pi(x, k) := (\Stone(\pi_{X \to Y})(x), \tilde \theta(x) k)
\end{equation}
is a $\CHProb_{\Gamma \times K^\op}$-morphism (in particular, it preserves the action of $\Gamma \times K^{\op}$).  Moreover, one has the identities
\begin{equation}\label{factor-1-a}
{\pi_{\Stone(Y) \times^\CHProb K \to \Stone(Y)}} \circ \tilde \Pi =_{\CHProb} \Stone(\pi) \circ {\pi_{\Stone(X) \times^\CHProb K \to \Stone(X) }} 
\end{equation}
in $\CHProb$ and
\begin{equation}\label{factor-2-a} 
{\pi}_{\Stone(Y) \times^\CHProb K \to K} \circ \tilde \Pi = (\tilde \theta \circ {\pi_{\Stone(X) \times^\CHProb K \to \Stone(X)}}) 
{\pi}_{\Stone(X) \times^\CHProb K \to K}
\end{equation}
in the group $\CondTilde_{\Stone(X) \rtimes_1^\CHProb K}(K)$, where $\pi_{Z \to W}$ denotes various canonical $\CHProb$-projections between the indicated spaces $Z,W$; see Figure \ref{fig:stone-prod}.

 \begin{figure}
    \centering
    \begin{tikzcd}
    & K & K \\
    & \Stone(X) \rtimes_1 K \arrow[u,"\pi"] \arrow[ur, "\pi \circ \tilde \Pi"] \arrow[r, "\tilde \Pi"] \arrow [d, "\pi"] \arrow [dl, "\tilde \theta \circ \pi"'] & \Stone(Y) \rtimes_{\tilde \rho} K \arrow[u, "\pi"'] \arrow[d, "\pi"] \\ 
    K & \Stone(X) \arrow[l, "\tilde \theta"] \arrow[r, "\Stone(\pi)"] & \Stone(Y)
    \end{tikzcd}
    \caption{A commutative diagram in $\CH$ illustrating \eqref{factor-1-a}, \eqref{factor-2-a}.  Various subscripts and superscripts have been suppressed to reduce clutter.  The three $\CH$-morphisms $\pi \circ \tilde \Pi, \tilde \theta \circ \pi, \pi$ from $\Stone(X) \rtimes_1 K$ to $K$ lie in the group $\CondTilde_{\Stone(X) \rtimes_1^\CHProb K}(K)$, and the first is the product of the second and third.}
    \label{fig:stone-prod}
\end{figure} 

If we let $\Pi: X \rtimes_1 K \to Y \rtimes_\rho K$ be the cast of $\tilde \Pi$ to $\ProbAlgGK$, we conclude (using \eqref{rep}) that $\Pi \colon X \rtimes_1 K \to Y \rtimes_\rho K$ is a $\ProbAlgGK$-morphism obeying the identities
\begin{equation}\label{factor-1}
{\pi_{Y \rtimes_\rho K \to Y}} \circ \Pi = \pi_{X \to Y} \circ {\pi_{X \rtimes_1 K \to X}} 
\end{equation}
in $\ProbAlg$ and
\begin{equation}\label{factor-2} \theta_{Y \rtimes_\rho K} \circ \Inc(\Pi) = (\theta \circ \Inc({\pi_{X \rtimes_1 K \to X}})) \theta_{X \rtimes_1 K}
\end{equation}
in the group $\Cond_{X \rtimes_1 K}(K)$, where $\pi_{Y \rtimes_\rho K \to Y}, \pi_{X \rtimes_1 K \to X}$ are the canonical $\ProbAlgG$-morphisms with the indicated domains and codomains, and $\theta_{Y \rtimes_\rho K} \in \Cond_{Y \rtimes_\rho K}(K)$, 
$\theta_{X \rtimes_1 K} \in \Cond_{X \rtimes_1 K}(K)$ are the canonical coordinate functions.

 \begin{figure}
    \centering
    \begin{tikzcd}
    & K & K \\
    & \Inc(X \rtimes_1 K) \arrow[u,"\theta_{X \rtimes_1 K}"] \arrow[ur, "\theta_{Y \rtimes_\rho K} \circ \Inc(\Pi)", near end] \arrow[r, "\Inc(\Pi)"] \arrow [d, "\Inc(\pi)"] \arrow [dl, "\theta \circ \Inc(\pi)"'] & \Inc(Y \rtimes_{\rho} K) \arrow[u, "\theta_{Y \rtimes_\rho K}"'] \arrow[d, "\Inc(\pi)"] \\ 
    K & \Inc(X) \arrow[l, "\theta"] \arrow[r, "\Inc(\pi)"] & \Inc(Y)
    \end{tikzcd}
    \caption{A commutative diagram in $\AbsMes$ illustrating \eqref{factor-1}, \eqref{factor-2}.  Various subscripts and superscripts have been suppressed to reduce clutter.  The three $\AbsMes$-morphisms $\theta_{Y \rtimes_\rho K} \circ \Pi, \theta \circ \Inc(\pi), \theta_{X \rtimes_1 K}$ from $\Inc(X \rtimes_1 K)$ to $K$ lie in the group $\Cond_{X \rtimes_1 K}(K)$, and the first is the product of the second and third.}
    \label{fig:stone-prod-abs}
\end{figure}

\section{Second step: extracting the Mackey range}

We are now ready to build more of the diagram in Figure \ref{fig:diagram}.
In Section \ref{product-system} we have constructed the $\ProbAlgGK$-morphism $\Pi \colon X \rtimes_1 K \to Y \rtimes_\rho K$.  Applying the invariant factor functor $\Inv$, we obtain a $\ProbAlgGK$-morphism $\Inv(\Pi) \colon \Inv(X \rtimes_1 K) \to \Inv(Y \rtimes_\rho K)$.  We now study the $\ProbAlgG$-systems $\Inv(X \rtimes_1 K)$ and $\Inv(Y \rtimes_\rho K)$.  For the former we have the following application of Proposition \ref{rel}:

\begin{proposition}\label{triv-erg} Let $X$ be a $\ProbAlgG$-system and $K$ be a compact Hausdorff group.  Then, using the skew-product construction in Definition \ref{abstract-ext}, we have a $\ProbAlgGK$-isomorphism
$$\Inv( X \rtimes_1 K ) \equiv \Inv(X) \rtimes_1 K.$$
\end{proposition}

\begin{proof}  Applying $\Stone$ to the $\ProbAlgG$-morphism from $X$ to $\Inv(X)$, we obtain a $\CHProb_\Gamma$-morphism from $\Stone(X)$ to $\Stone(\Inv(X))$, which then gives a $\CHProb_{\Gamma \times K^{op}}$-morphism from $\Stone(X) \rtimes_1 K$ to $\Stone(\Inv(X)) \rtimes_1 K$.  One can easily check that the pullback of any element of $L^\infty(\Stone(\Inv(X)) \rtimes_1 K)$ in $\Stone(X) \rtimes_1 K$ is invariant under the $\Gamma$ action, and thus can be identified with an element of $L^\infty(\Inv(\Stone(X) \rtimes_1 K))$.  Equivalently, we see that $L^\infty(\Inv(X) \rtimes_1 K)$ can be viewed as a subalgebra of $L^\infty(\Inv(X \rtimes_1 K))$, with both spaces being identifiable in turn with subalgebras of $L^\infty(X \rtimes_1 K)$. If we can show that 
\begin{equation}\label{lii}
L^\infty(\Inv(X) \rtimes_1 K) = L^\infty(\Inv(X \rtimes_1 K)),
\end{equation}
then on taking idempotents we obtain a one-to-one correspondence between the elements of the $\SigmaAlg$-algebra of $\Inv(X) \rtimes_1 K$ with the $\SigmaAlg$-algebra of $\Inv(X \rtimes_1 K)$, which then easily leads to the required $\ProbAlgGK$-isomorphism.

It remains to establish the identity \eqref{lii}.  Taking $L^2$ closures, it suffices to show that
$$L^2(\Inv(X) \rtimes_1 K) = L^2(\Inv(X \rtimes_1 K)),$$
viewing both spaces as subspaces of $L^2(X \rtimes_1 K)$.   If this is not the case, then there is a non-zero $f \in L^2(\Inv(X \rtimes_1 K))$ that is orthogonal to $L^2(\Inv(X) \rtimes_1 K)$. The latter space contains all products $gh$ with $g \in L^\infty(K)$ and $h \in L^2(\Inv(X))$ (where we embed $L^\infty(K)$ into $L^\infty(\Inv(X) \rtimes_1 K)$ and $L^2(\Inv(X))$ into $L^2(\Inv(X) \rtimes_1 K)$ in the obvious fashion).  We conclude that $fg$ is orthogonal to $L^2(\Inv(X))$ for all $g \in L^\infty(K)$.  On the other hand, $g \in L^\infty(\Inv(X \rtimes_1 K))$ and hence $fg \in L^2(\Inv(X \rtimes_1 K))$.  Applying Proposition \ref{rel}, we see that $fg$ is therefore also orthogonal to $L^2(X)$.  In particular, $f$ is orthogonal to the (algebraic) tensor product $L^2(X) \otimes L^\infty(K) = L^2(\Stone(X)) \otimes L^\infty(K)$, but this is dense in $L^2(X \rtimes_1 K) = L^2(\Stone(X) \times^\CHProb K)$, hence $f$ is orthogonal to itself, a contradiction.
\end{proof}

From Proposition \ref{triv-erg} and the ergodicity of $X$ we have the $\ProbAlgGK$-isomorphisms
$$ \Inv(X \rtimes_1 K) \equiv \Inv(X) \rtimes_1 K \equiv \mathrm{pt} \rtimes_1 K$$
where $\mathrm{pt}$ is a point in $\ProbAlgG$.  By chasing the definitions, $\mathrm{pt} \rtimes_1 K$ is $\ProbAlg_{\Gamma \times K^{\op}}$-isomorphic to the $\ConcProb_{\Gamma \times K^\op}$-system $K$ endowed with the concrete action
$$ T_K^{(\gamma,k_0^\op)}(k) := kk_0,$$
so we now have a $\ProbAlgGK$-morphism from $X \rtimes_1 K$ to $K$, which after chasing the definitions is seen to agree (in say $\AbsProb$, after applying $\Inv$) with the vertical coordinate $\theta_{X \rtimes_1 K}$.  We also have a $\ProbAlgGK$-morphism from $K$ to $\Inv(Y \rtimes_\rho K)$.  Applying a forgetful functor, this is also a $\ProbAlg_{K^\op}$-morphism $\pi$, thus $(\Inv(Y \rtimes_\rho K),\pi)$ is now a $\ProbAlg_{K^\op}$-factor of $K$.  We can classify such factors:

\begin{lemma}[Factors of $K$]\label{piw}  Every $(\ProbAlg_{K^\op},\pi)$-factor $Z$ of $K$ is $(K \downarrow \ProbAlg_{K^\op})$-isomorphic to $(H \backslash K, \pi_{H \backslash K})$ for some compact subgroup $H$ of $K$, where $K^\op$ acts on the $\ConcProb$-space $H \backslash K$ by $T^{k_0^\op}_{H \backslash K}(Hk) := Hkk_0$, and $\pi \colon K \to H \backslash K$ is the quotient map (with $H \backslash K$ and $\pi$ both casted to $\ProbAlg_{K^\op}$, and $H \backslash K$ equipped with Haar measure $\Haar_{H \backslash K}$).
\end{lemma}

We remark that a version of this lemma for separable (or equivalently, metrizable) $K$ was implicitly given in \cite{mackey}.

\begin{proof}  Using the pullback map associated to $\pi$ (or $\Stone(\pi)$), we can identify $L^2(Z)$ with a closed subspace of $L^2(K)$, which is then invariant under the right multiplication action of $K$.  As Haar measure is a Baire--Radon measure (see Theorem \ref{rrt}), we see from Urysohn's lemma or Lusin's theorem that $C(K)$ is dense in $L^2(K)$. By Young's inequality, this latter space is also closed under the convolution operation
$$ f*g(k) := \int_K f(k') g((k')^{-1} k)\ d\Haar_K(k').$$
If $\psi \in C(K)$ has total integral $\int_K \psi\ d\Haar_K = 1$ equal to one, then $f*\psi$ lies in the closed convex hull of all right-translates of $f$ for any $f \in L^2(K)$; this is clear for continuous $f$ by uniform continuity\footnote{While we do not assume $K$ to be metrizable, it is still a uniform space, so the usual theory of uniform continuity still applies.}, and the general case follows by a density argument.  In particular, we see from the right-invariance of $L^2(Z)$ that the convolution operator $f \mapsto f * \psi$ maps $L^2(Z)$ to $L^2(Z)$.  From the uniform continuity of $\psi$ we also see that $f*\psi$ is continuous.  One can construct a net $\psi_U$ of approximate identities supported on arbitrary neighbourhoods $U$ of the identity such that $f*\psi_U$ converges in $L^2(K)$ to $f$ for every $f \in L^2(K)$ (again, this is easiest to verify first for continuous $f$, with the general case then following by density), so in particular we see that every function in $L^2(Z)$ can be written as the  limit of functions in the space $\mathcal{A} := C(K) \cap L^2(Z)$.  In other words, $\mathcal{A}$ is dense in $L^2(Z)$.

Let $H$ be the left symmetry group of ${\mathcal A}$, that is to say $H$ is the collection of all $k_0 \in K$ such that $f(k_0 k) = f(k)$ for all $f \in {\mathcal A}$ and $k \in K$.  Clearly $H$ is a compact subgroup of $K$, and $\mathcal{A}$ may be viewed as a subspace of $C(H\backslash K) \subset L^2(H \backslash K)$.  It is also invariant under right translations by $K$.   From this and the definition of $H$ we see that $\mathcal{A}$ is a unital algebra that separates points, hence by the Stone--Weierstra{\ss} theorem it is dense in  $C(H\backslash K)$ in the uniform norm, hence dense in $L^2( H \backslash K )$.  Thus the $L^2(K)$ closure of $\mathcal{A}$ can be identified with $L^2( H \backslash K)$.  But this closure was already found to equal $L^2(Z)$, thus
$$ L^2(Z) \equiv L^2( H \backslash K ).$$
Inspecting the idempotent elements of both sides, we see that $Z$ can be identified in $\ProbAlg$ with $H \backslash K$; using the Koopman action one can promote this to a $\ProbAlg_{K^\op}$-identification, and the claim follows.
\end{proof}

Applying this lemma to our current situation, we conclude that there is a compact subgroup $H$ of $K$ such that $\Inv(Y \rtimes_\rho K)$ is $\ProbAlg_{K^\op}$-isomorphic to $H \backslash K$ for some compact subgroup $H$ of $K$; we can promote this to a $\ProbAlgGK$-isomorphism by letting $\Gamma$ act trivially on $H \backslash K$ (after first applying $\Stone$ to define the action concretely if desired).  We thus create a $\ProbAlgGK$-morphism  $\psi \colon Y \rtimes_\rho K \to H\backslash K$, completing the right-hand portion of the commutative diagram in Figure \ref{fig:diagram}.

\section{Third step: straightening the vertical coordinate}

Recall from Definition \ref{abstract-ext} that the $\ProbAlgG$-system $X$ is equipped with a vertical coordinate $\theta \in \Cond_X(K)$.  In this section we establish the following result, which uses the various morphisms in Figure \ref{fig:diagram} to straighten this vertical coordinate by using left multiplication from the factor $Y$ to move it into the Mackey range $H$.  The lifting result in Theorem \ref{lift} will play a critical role in this step.

\begin{proposition}[Straightening the vertical coordinate]\label{perfect}  There exists $\Phi \in \Cond_Y(K)$ such that $(\Phi \circ \Inc(\pi_{X \to Y})) \theta \in \Cond_X(H)$ (using the group law in $\Cond_X(K)$). 
\end{proposition}

\begin{proof}  As $\psi \colon Y \rtimes_\rho K \to H \backslash K$ is a $\ProbAlgGK$-morphism, it is also a $\ProbAlg_{K^\op}$-morphism, thus for any $k_0^{\op} \in K^{\op}$ we have
$$ T_{H \backslash K}^{k_0^\op} \circ \psi = \psi \circ T_{Y \rtimes_\rho K}^{k_0^\op}.$$
If we let $\tilde \psi \in \Cond_{Y \rtimes_\rho K}(H \backslash K)$ be the conditional element of $H \backslash K$ defined by
$$ \tilde \psi := \iota \circ \Inc(\psi)$$
where $\iota \colon \Inc(H \backslash K) \to H \backslash K$ is the canonical $\AbsProb$-inclusion, we conclude that
$$ \tilde \psi k_0 = (T_{Y \rtimes_\rho K}^{k_0^\op})^* \tilde \psi$$
in $\Cond_{Y \rtimes_\rho K}(H \backslash K)$, where we view $k_0$ as an element of $\Cond_{Y \rtimes_\rho K}(K)$ which acts on the right on $\Cond_{Y \rtimes_\rho K}(H \backslash K)$.  Meanwhile, the standard vertical coordinate $\theta_{Y \rtimes_\rho K} \in \Cond_{Y \rtimes_\rho K}(K)$ obeys the similar identity
$$ \theta_{Y \rtimes_\rho K} k_0 = (T_{Y \rtimes_\rho K}^{k_0^\op})^* \theta_{Y \rtimes_\rho K},$$
in $\Cond_{Y \rtimes_\rho K}(K)$, as can be seen by first observing the identity
$$ \pi_K k_0 =_{\CH} \left(T_{\Stone(Y) \rtimes^{\CHProb_\Gamma}_{\tilde \rho} K}^{k_0^\op}\right)^* \pi_K$$
in the group $\Hom_\CH( \Stone(Y) \rtimes_\rho^{\CHProb_\Gamma} K \to K )$, with $\pi_K \colon \Stone(Y) \rtimes_\rho^{\CHProb_\Gamma} K \to K$ the canonical $\CH$-projection, and then abstracting.  Thus if we introduce the $\Cond_{Y \rtimes_\rho K}(H \backslash K)$-element
\begin{equation}\label{psi1-def}
 \psi_1 := \tilde \psi \theta_{Y \rtimes_\rho K}^{-1}
\end{equation}
we have the identity
$$ (T_{Y \rtimes_\rho K}^{k_0^\op})^* \psi_1 = \psi_1$$
in $\Cond_{Y \rtimes_\rho K}(H \backslash K)$ for all $k_0^\op \in K^\op$.  If we were in a concrete setting we could immediately imply that $\psi_1$ descends from $Y \rtimes_\rho K$ to $Y$.  In the current abstract setting we have to be slightly more\footnote{The problem here is that the concrete model $\Stone(Y) \rtimes^{\CHProb_\Gamma}_{\tilde \rho} K$ of $Y \rtimes_\rho K$ does not agree with the \emph{canonical} model $\Stone(Y \rtimes_\rho K)$, and the nature of the $K^{\op}$ action on the latter model is not completely obvious.  Hence one cannot simply apply the canonical model functor to study the $K^{\op}$ action in a concrete setting.} careful.  For any $E \in (H \backslash K)_\SigmaAlg$, we can view $1_{(\psi_1)_\SigmaAlg(E)}$ as an element of $L^2(Y \times^{\ProbAlg} K) = L^2(\Stone(Y) \times^\ConcProb K)$ which is invariant under the action of $K^\op$.  By Lemma \ref{trans}, such elements must arise from $L^2(\Stone(Y)) = L^2(Y)$, hence $(\psi_1)_\SigmaAlg(E)$ lies in $\Inc(Y)_\SigmaAlg$ (embedded into $\Inc(Y \rtimes_\rho K)_\SigmaAlg$ in the obvious fashion).  Thus we have $\psi_1 = \psi_0 \circ \Inc(\pi_Y)$ for some $\psi_0 \in \Cond_Y(H \backslash K)$, where $\pi_Y \colon Y \rtimes_\rho K \to Y$ is the canonical $\ProbAlg$-projection. Inserting this into \eqref{psi1-def}, we conclude that
$$
\psi = (\psi_0 \circ \Inc(\pi_Y)) \theta_{Y \rtimes_\rho K}$$
in $\Cond_{Y \rtimes_\rho K}(H \backslash K)$.  Applying Theorem \ref{lift}, one can lift $\psi_0 \in \Cond_Y(H \backslash K)$ to a conditional element $\Phi \in \Cond_Y(K)$, thus $\psi_0 = H \Phi$ where we view $H$ as the identity element of $H \backslash K$ (and naturally identified with an element of $\Cond_Y(H \backslash K)$).  Thus
$$ \psi = H (\Phi \circ \Inc(\pi_Y)) \theta_{Y \rtimes_\rho K}$$
in $\Cond_{Y \rtimes_\rho K}(H \backslash K)$.  Performing a base change to $X \rtimes_1 K$ using Figure \ref{fig:diagram} and \eqref{Pi-def} (or \eqref{factor-1}, \eqref{factor-2}), we conclude that
\begin{align*}
 H \theta_{X \rtimes_1 K} &= H (\Phi \circ \Inc(\pi_Y \circ \Pi)) (\theta_{Y \rtimes_\rho K} \circ \Inc(\Pi)) \\
&= H (\Phi \circ \Inc(\pi_{X \to Y} \circ \pi_X)) (\theta \circ \Inc(\pi_X)) \theta_{X \rtimes_1 K} 
\end{align*}
in $\Cond_{X \rtimes_1 K}(H \backslash K)$, where $\pi_X \colon X \rtimes_1 K \to X$ is the canonical $\ProbAlg$-projection. We can rearrange this as
$$ H = ( H (\Phi \circ \Inc(\pi_{X \to Y}) \theta ) \circ \Inc(\pi_X)$$
in  $\Cond_{X \rtimes_1 K}(H \backslash K)$.  By Lemma \ref{base-change}, we conclude that
$$ H = H (\Phi \circ \Inc(\pi_{X \to Y}) \theta ) $$
in $\Cond_X(H \backslash K)$.  By Theorem \ref{lift}, we conclude that
$$\Phi \circ \Inc(\pi_{X \to Y}) \theta \in \Cond_X(H)$$
as required.
\end{proof}

\section{Fourth step: A Fubini--Tonelli calculation}\label{fourth-step}

By Proposition \ref{perfect}, we can find $\Phi \in \Cond_Y(K)$ such that the modified vertical coordinate
\begin{equation}\label{theta-star}
\theta_* := (\Phi \circ \Inc({\pi_{X \to Y}})) \theta 
\end{equation}
lies in $\Cond_X(H)$.  We will now show that the maps $\pi_{X \to Y}, \theta_*$ generate a $\ProbAlgG$-isomorphism between $X$ and a skew-product $Y \rtimes_{\rho_*} H$ for some cocycle $\rho_*$, which will establish Theorem \ref{mackey-uncountable}(i).

By \eqref{rep}, we can canonically model $\theta_*$ by a concrete $\CH$-morphism $\tilde \theta \in \CondTilde_{\Stone(X)}(H)$.  We now have the following key computation:

\begin{lemma}[Computation of integral]\label{integral-comp} We have
$$ \int_{\Stone(X)} (f \circ \Stone(\pi_{X \to Y})) (g \circ \tilde \theta) = \left(\int_{\Stone(Y)} f\right) \left(\int_H g\right)$$
for all $f \in C(\Stone(Y))$ and $g \in C(H)$. 
\end{lemma}

\begin{proof}  By the Tietze extension theorem we can work with $g \in C(K)$ instead of $g \in C(H)$.  Using approximations to the identity and uniform continuity, one can approximate $g$ to arbitrary accuracy in the uniform topology by convolutions 
$$ h \mapsto \int_K g(hk) \varphi(k)\ d\Haar_K(k)$$
for $\varphi, g \in C(K)$.  Thus by Fubini's theorem it suffices to show that
\begin{equation}\label{lam}
\begin{split}
&\int_{\Stone(X)} \int_K f(\Stone(\pi_{X \to Y}(\tilde x)) g( \tilde \theta(\tilde x) k) \varphi(k)\ d\Haar_K(k) \ d\mu_{\Stone(X)}(\tilde x) \\
&\quad = \left(\int_{\Stone(Y)} f\right) \left(\int_H \int_K g(hk) \varphi(k)\  d\Haar_K(k) d\Haar_H(h)\right)
\end{split}
\end{equation}
for $f \in C(\Stone(Y))$ and $\varphi,g \in C(K)$.
By splitting $\varphi$ into its average
$$ k \mapsto \int_H \varphi(hk)\ d\Haar_H(h)$$
and its mean zero part
$$ k \mapsto \varphi(k) - \int_H \varphi(hk)\ d\Haar_H(h),$$
it suffices to prove \eqref{lam} in two cases: firstly when we have the left-invariance
\begin{equation}\label{left}
\varphi(hk) = \varphi(k)
\end{equation}
for all $h \in H, k \in K$
and secondly when we have the mean zero condition
\begin{equation}\label{mean-zero}
\int_H \varphi(hk)\ d\Haar_H(h) = 0
\end{equation}
for all $k \in K$.

First suppose that we have the mean-zero condition \eqref{mean-zero}.  Then from the Fubini--Tonelli theorem and the invariance of Haar measure we see that
$$ \int_H \int_K g(hk) \varphi(k)\ d\Haar_H(h) d\Haar_K(k) = 0,$$
so that the right-hand side of \eqref{lam} vanishes.  Meanwhile, by \eqref{Pi-def}, the Fubini--Tonelli theorem, and the invariance of Haar measure, the left-hand side of \eqref{lam} may be written as
$$ \int_{\Stone(X) \rtimes_1 K} (f \otimes g)(\tilde \Pi(\tilde x, k)) \varphi(k)\ d\mu_{\Stone(X) \rtimes_1 K}(\tilde x,k)$$
where $f \otimes g \in \CH(\Stone(Y) \rtimes_\rho K)$ is defined by
$$(f \otimes g)(\tilde y,k) := f(\tilde y) g( k).$$
In particular, the function $(f \otimes g) \circ \tilde \Pi$ lies in $L^2(Y \rtimes_\rho K)$.  Meanwhile, from \eqref{mean-zero} we see that $(\tilde x,k) \mapsto \varphi(k)$ lies in $L^2(K) \ominus L^2(H \backslash K) = L^2(\Inv(X \rtimes_1 K)) \ominus L^2(\Inv(Y \rtimes_\rho K))$.  From Proposition\ref{rel} we conclude that these two functions are orthogonal, giving \eqref{lam} in this case.

Now suppose we are in the invariant case \eqref{left}.  Then from the Fubini--Tonelli theorem and the invariance of Haar measure, the right-hand side of \eqref{lam} simplifies to
$$ \left(\int_{\Stone(Y)} f\right) \left(\int_K g \varphi\right)$$
and the left-hand side similarly simplifies to
$$ \left(\int_{\Stone(X)} f \circ \Stone(\pi_{X \to Y})\right) \left(\int_K g\varphi\right)$$
and the claim follows since $\Stone(\pi_{X \to Y})$ is a $\CHProb$-morphism.
\end{proof}

From \eqref{theta-star}, \eqref{theta-gam-abstract} we have the identity
\begin{equation}\label{thet-id}
\theta_* \circ T_X^\gamma = (\rho_{*,\gamma} \circ \Inc(\pi_{X \to Y})) \theta_*
\end{equation}
in $\Cond_X(K)$, where $\rho_{*,\gamma} \in \Cond_Y(K)$ is the cocycle cohomologous to $\rho_\gamma$ defined by the formula
$$ \rho_{*,\gamma} := (\Phi_* \circ T_Y^\gamma) \rho_\gamma \Phi_*^{-1}.$$
Since $\theta_* \in \Cond_X(H)$, we see from \eqref{thet-id} that
$$ \rho_{*,\gamma} \circ \Inc(\pi_{X \to Y}) \in \Cond_X(H).$$
By Lemma \ref{base-change}, we conclude that
$$ \rho_{*,\gamma} \in \Cond_Y(H).$$
Since $\rho_{*,\gamma}$ can be easily verified to obey the cocycle equation \eqref{cocycle}, we conclude that $\rho_* := (\rho_{*,\gamma})_{\gamma \in \Gamma}$ is an $H$-valued $\ProbAlgG$-cocycle.  
Applying Lemma \ref{integral-comp} and Proposition \ref{skew-crit}, we conclude that $X$ is equivalent as an extension of $Y$ to $Y \rtimes^{\ProbAlgG}_{\rho_*} H$. 
This completes the proof of Theorem \ref{mackey-uncountable}(i).

\section{Fifth step: extending ergodic homogeneous extensions to ergodic group extensions}\label{sixth}

Having established part (i) of Theorem \ref{mackey-uncountable}, we now work on part (ii), where the vertical coordinate  $\theta \colon X \to K/L$ now takes values in a group quotient $K/L$ for some closed subgroup $L$ of $K$.   The main step is encapsulated in the following theorem (cf. \cite[Corollary 3.27]{glasner2015ergodic}), which one can also view as a variant of Theorem \ref{lift}.

\begin{theorem}[Extending homogeneous extensions to group extensions]\label{ext-thm}  Let $\Gamma$ be a group, $Y$ be an ergodic $\ProbAlgG$-system, $K$ be a compact Hausdorff group, $L$ be a closed subgroup of $K$, and $(X, \pi_{X \to Y}, \theta, \rho)$ be an ergodic $\ProbAlgG$-homogeneous extension of $Y$ by $K/L$.
\begin{itemize}
    \item [(i)]  There exists a $\ProbAlgG$-group extension $(X', T_{X'}, \pi_{X' \to Y}, \theta', \rho')$ of $(Y,T_Y)$ by $K$ and a $\ProbAlgG$-factor map $\pi_{X' \to X} \colon (X',T_{X'}) \to (X,T)$ such that the diagram
    \begin{center}
\begin{tikzcd}
    \Inc(X') \arrow[r,"\theta'"] \arrow[d, "\Inc(\pi_{X' \to X})"] \arrow[dd, bend right, "\Inc(\pi_{X' \to Y})"',shift right=0.75ex] & K \arrow[d, "\pi_{K/L}"]\\
    \Inc(X) \arrow[r,"\theta"] \arrow[d,"\Inc(\pi_{X \to Y})"] & K/L\\
    \Inc(Y)
\end{tikzcd}
\end{center}
    commutes in $\AbsMes$, where ${\pi}_{K/L} \colon K \to K/L$ is the canonical projection.
    \item[(ii)]  In part (i), we can take the $\ProbAlgG$-system $(X',T_{X'})$ to be ergodic.
\end{itemize}
\end{theorem}

We begin with part (i).
We first apply the functor $\Stone$ and \eqref{rep} to canonically model the $\ProbAlgG$-homogeneous extension $(X, T_X, \pi_{X \to Y}, \theta, \rho)$ of $(Y,T_Y)$ by $K/L$ by the  $\CHProb_\Gamma$-homogeneous extension $(\tilde X, T_{\tilde X}, \pi_{\tilde X \to \tilde Y}, \tilde \theta, \tilde \rho)$ of $(\tilde Y,T_{\tilde Y}) := \Stone(Y,T_Y)$ by $K/L$, where $(\tilde X, T_{\tilde X}) := \Stone(X,T_X)$, $\pi_{\tilde X \to \tilde Y} := \Stone(\pi_{X \to Y})$ and the $\CH$-morphism $\tilde \theta \colon \tilde X \to K/L$ and the $K$-valued $\CHProb_\Gamma$-cocycle $\tilde \rho$ on $\tilde X$ are the canonical concrete models of $\theta, \rho$ respectively.  We then set $\tilde X'$ to be the compact subset of $\tilde X \times^\CH K$ defined by
$$ \tilde X' := \{ (\tilde x, k) \in \tilde X \times^\CH K: \pi_{K/L}(k) = \tilde \theta(\tilde x) \}.$$
This is a $\CH$-space.
We have the $\CH$-morphisms $\pi_{\tilde X' \to \tilde X}, \pi_{\tilde X \to \tilde Y'}$, $\tilde \theta'$ defined by
\begin{align*}
\pi_{\tilde X' \to \tilde X}(\tilde x, k) &:= \tilde x\\
\pi_{\tilde X' \to \tilde Y}(\tilde x, k) &:= \pi_{\tilde X \to \tilde Y}(\tilde x) \\
\tilde \theta'(\tilde x, k) &:= k
\end{align*}
for $(\tilde x, k) \in \tilde X'$.  Then we have the commuting diagram
    \begin{center}
\begin{tikzcd}
    \tilde X' \arrow[r,"\tilde \theta'"] \arrow[d, "\pi"] \arrow[dd, bend right, "\pi"'] & K \arrow[d, "\pi"]\\
    \tilde X \arrow[r,"\tilde \theta"] \arrow[d,"\pi"] & K/L\\
    \tilde Y
\end{tikzcd}
\end{center}
in $\CH$, where we suppress subscripts on the various projections $\pi$ for brevity.

We can add (topological) dynamics to the left column of this diagram by introducing the $\CH$-action $T_{\tilde X'}$ of $\Gamma$ by the formula
$$ T^\gamma_{\tilde X'}(\tilde x, k) := (T^\gamma_{\tilde X}(\tilde x), \tilde \rho_\gamma(\pi_{\tilde X \to \tilde Y}(\tilde x)) k).$$
From \eqref{theta-gam-abstract} (transferred to the concrete model) we see that these maps are well defined as $\CH$-morphisms of $\tilde X'$.  
As $\tilde \rho$ is a $K$-valued $\CHProb$-cocycle, all the morphisms on the left column of this diagram can be promoted to $\CH_\Gamma$ morphisms.

Now we add probability theory to the left column, by defining a measure $\overline{\mu}_{\tilde X'}$ on $\tilde X'$ by the formula
$$ \int_{\tilde X'} F\ d\overline{\mu}_{\tilde X'} := \int_{\tilde X} \left(\int_{\tilde \theta(\tilde x)} F(\tilde x, k)\ d\Haar_{\tilde \theta(\tilde x)}(k)\right)\ d\mu_{\tilde X}(\tilde x)$$
for any $F \in C(\tilde X')$, where $\Haar_{\tilde \theta(\tilde x)}$ is Haar measure on the left coset $\tilde \theta(\tilde x)$ of $L$.  By Theorem \ref{rrt} this uniquely defines a probability measure on $\tilde X'$, and from Fubini's theorem and the invariance property of Haar measure (as well as further application of Theorem \ref{rrt}) we see that this measure is preserved by the action $T_{\tilde X'}$, and pushes down to $\mu_{\tilde X}$ under $\pi_{\tilde X' \to \tilde X}$.  Thus if we equip $\tilde X'$ with the measure $\overline{\mu}_{\tilde X'}$, the left column of the above diagram can now be promoted to $\CHProb_\Gamma$-morphisms.  Casting to $\ProbAlgG$, applying $\Inc$, and then using the canonical $\AbsMes$-inclusions of $\Inc(\tilde X'), \Inc(\tilde X)$ into $\tilde X', \tilde X$ respectively, we obtain part (i) of the theorem.

Now we turn to part (ii).  Consider the collection of all possible probability measures $\mu_{\tilde X'}$ on $\tilde X'$ for which $(\tilde X, \mu_{\tilde X'}, T_{\tilde X'}, \pi_{\tilde X' \to \tilde X})$ is a $\CHProb_\Gamma$-extension of $\tilde X$, or equivalently (by Theorem \ref{rrt}) one has the identities
\begin{equation}\label{mu-1} 
\int_{\tilde X'} f \circ \pi_{\tilde X' \to \tilde X}\ d\mu_{\tilde X'}
= \int_{\tilde X} f
\end{equation}
and
\begin{equation}\label{mu-2}
\int_{\tilde X'} F \circ T^\gamma_{\tilde X'}\ d\mu_{\tilde X'} =
\int_{\tilde X'} F \ d \mu_{\tilde X'}
\end{equation}
for all $f \in C(\tilde X)$, $F \in C(\tilde X')$. By Theorem \ref{rrt} and Tychonoff's theorem, we may identify this collection with a closed convex subset of $\C^{C(\tilde X')}$; by the preceding construction, this subset is non-empty.  Thus by the Krein--Milman theorem, we may find a probability measure $\mu_{\tilde X'}$ in this collection which is an extreme point of the convex set.  If we can show that the system $(\tilde X', \mu_{\tilde X'}, T_{\tilde X'})$ is ergodic, then by repeating the argments used to conclude (i) we obtain (ii).

It remains to establish ergodicity.  If for contradiction one does not have ergodicity, there must exist $0<p<1$ and a set $E \in \tilde X'_{\SigmaAlg}$ of measure $\mu_{\tilde X'}(E) = p$ such that $E$ and $T^\gamma_{\tilde X'} E$ agree up to $\mu_{\tilde X'}$-null sets for all $\gamma \in \Gamma$.  The measure $1_E \mu_{\tilde X'}$ is then $T_{\tilde X'}$-invariant and of total mass $p$, so the pushforward
$(\pi_{\tilde X' \to \tilde X})_* (1_E \mu_{\tilde X'})$ is a  $T_{\tilde X}$-invariant measure of total mass $p$ that is absolutely continuous with respect to $\mu_{\tilde X}$.  Hence we must have
$$ (\pi_{\tilde X' \to \tilde X})_* (1_E \mu_{\tilde X'}) = p \mu_{\tilde X}.$$
If we then split $\mu_{\tilde X'} = p \mu^1_{\tilde X} + (1-p) \mu^2_{\tilde X}$, where
$$ \mu^1_{\tilde X} := \frac{1}{p} 1_E \mu_{\tilde X}; \quad \mu^2_{\tilde X} := \frac{1}{1-p} 1_{\tilde X \backslash E} \mu_{\tilde X}$$
one can verify that $(\tilde X, \mu^i_{\tilde X'}, T_{\tilde X'}, \pi_{\tilde X' \to \tilde X})$ are $\CHProb_\Gamma$-extensions of $\tilde X$ for $i=1,2$ with $\mu^i_{\tilde X'} \neq \mu_{\tilde X'}$, contradicting the extremality of $\mu_{\tilde X'}$.  Hence no such $E$ exists, and the proof of Theorem \ref{ext-thm} is complete.

\section{Sixth step: quotienting out the group skew-product}

We are now ready to conclude the proof of Theorem \ref{mackey-uncountable}(ii). Let $\Gamma$ be a group, $Y$ be an ergodic $\ProbAlgG$-system, $K$ be a compact group, $L$ be a closed subgroup of $K$, and $(X, \pi_{X \to Y}, \theta, \rho)$ be an ergodic $\ProbAlgG$-homogeneous extension of $Y$ by $K/L$.  Our task is to show that the extension $(X, \pi_{X \to Y})$ is $\ProbAlgG$-equivalent to a $\ProbAlgG$-homogeneous skew-product $Y \rtimes_\rho H/M$ for some $H$ a compact subgroup of $K$ and $M$ a compact subgroup of $H$.   

By Theorem \ref{ext-thm}, we can find an ergodic $\ProbAlgG$-group extension $(X', \pi_{X' \to Y}, \theta', \rho')$ of $Y$ by $K$ and an a $\ProbAlgG$-factor map $\pi_{X' \to X} \colon X' \to X$ such that the diagram in that theorem commutes.  From Theorem \ref{mackey-uncountable}(i), the extension $(X', \pi_{X' \to Y})$ is equivalent to a $\ProbAlgG$-group extension $Y \rtimes_{\rho_*} H$ for some compact subgroup $H$ of $K$ and some $H$-valued cocycle $\rho_*$.  Thus we may assume without loss of generality that $X' = Y \rtimes_{\rho_*} H$.  The vertical coordinate $\theta'$ need not agree with the standard vertical coordinate $\theta_{Y \rtimes_{\rho_*} H}$, but an inspection of the proof of Theorem \ref{mackey-uncountable}(i) reveals that the two coordinates are related by the equation
$$ (\Phi \circ \Inc(\pi_Y)) \theta' = \theta_{Y \rtimes_{\rho_*} H}$$
in $\Cond_{Y \times_{\rho_*} H}(K)$ for some $\Phi \in \Cond_Y(K)$, where $\pi_Y \colon Y \rtimes_{\rho_*} H \to Y$ is the canonical $\ProbAlg$-morphism. Applying the projection $\pi_{K/L} \colon K \to K/L$, we conclude in particular that
$$ (\Phi \circ \Inc(\pi_Y)) (\theta \circ \Inc(\pi_{X' \to X})) = \pi_{K/L}( \theta_{Y \times_{\rho_*} H} )$$
using the action of $\Cond_{Y \times_{\rho_*} H}(K)$ on $\Cond_{Y \times_{\rho_*} H}(K/L)$.  Equivalently, one has
\begin{equation}\label{geo}
\theta_* \circ \Inc(\pi_{X' \to X}) = \pi( \theta_{Y \times_{\rho_*} H}) 
\end{equation}
where $\theta_* \in \Cond_X(K/L)$ is defined by
$$\theta_* := (\Phi \circ \Inc(\pi_{X \to Y})) \theta.$$
Write $M := L \cap H$, then there is a natural monomorphism of $H/M$ into $K/L$.  The right-hand side of \eqref{geo} lies in $\Cond_{Y \times_{\rho_*} H}(H/M)$, hence (by Lemma \ref{base-change}) we have
$$\theta_* \in \Cond_X(H/M).$$
Thus we now have the commutative diagram
\begin{center}
\begin{tikzcd}
    \Inc(Y \rtimes_{\rho_*} H) \arrow[r,"\theta_{Y \rtimes_{\rho_*} H}"] \arrow[d, "\Inc(\pi_{X' \to X})"] \arrow[dd, bend right, shift right=1.5ex, "\Inc(\pi_{X' \to Y})"'] & H \arrow[d, "\pi_{H/M}"]\\
    \Inc(X) \arrow[r,"\theta_*"] \arrow[d,"\Inc(\pi_{X \to Y})"] & H/M\\
    \Inc(Y)
\end{tikzcd}
\end{center}
in $\AbsMes$, where $\pi_{H/M}$ denotes the projection from $H$ to $H/M$.  Note that the vertical morphisms are in fact $\AbsProb$-morphisms, as is the morphism from $\Inc(Y \rtimes_{\rho_*} H)$ to $H$, hence $\theta_*$ can also be promoted to an $\AbsProb$-morphism.  We may now cast the above diagram to $\ProbAlg$ to obtain

\begin{center}
\begin{tikzcd}
    Y \rtimes_{\rho_*} H \arrow[r,"\theta_{Y \rtimes_{\rho_*} H}"] \arrow[d, "\pi_{X' \to X}"] \arrow[dd, bend right, shift right=1.5ex, "\pi_{X' \to Y}"'] & H \arrow[d, "\pi_{H/M}"]\\
    X \arrow[r,"\theta_*"] \arrow[d,"\pi_{X \to Y}"] & H/M\\
    Y
\end{tikzcd}.
\end{center}

Now we verify the hypotheses of Proposition \ref{skew-crit}.
Suppose that $f \in L^\infty(Y)$ and $g \in C(H/M)$, and consider the abstract integral
$$ \int_X (\pi_{X \to Y}^* f) (g \circ \theta_*)$$
where the $\AbsMes$-morphism $g \circ\theta_* \colon \Inc(X) \to \C$ may be viewed as an element of $L^\infty(X)$.  Pulling back to $Y \rtimes_{\rho_*} K$ using the above commutative diagram, this becomes
$$ \int_{Y \rtimes_{\rho_*} H} (\pi_{X' \to Y}^* f) (g \circ \pi_{H/M} \circ \theta_{Y \rtimes_{\rho_*} H}).$$
From Definition \ref{abstract-ext}, we may write this as
$$ \int_{\Stone(Y) \rtimes_{\tilde \rho_*} H} f(\tilde y) g(hM) d\mu_{\Stone(Y)}(\tilde y) d\Haar_H(h)$$
which by Fubini's theorem simplifies to
$$ \left( \int_{\Stone(Y)} f\right) \left( \int_{H/M} g \right).$$
 Applying Proposition \ref{skew-crit}, we conclude that $X$ is equivalent as an extension of $Y$ to $Y \rtimes^{\ProbAlgG}_{\rho_*} H/M$, concluding the proof of Theorem \ref{mackey-uncountable}(ii). 

\section{Recovering the countable Mackey--Zimmer theorem from the uncountable theorem}\label{implication}

We now use Theorem \ref{mackey-uncountable} to establish Theorem \ref{mackey-countable}.  We will just establish part (ii) of this theorem, as part (i) is proven similarly.

Let the hypotheses be as in Theorem \ref{mackey-countable}(ii). Thus we have an at most countable group $\Gamma$, a standard Lebesgue $\ConcProb_\Gamma$-system $(\tilde Y, T_{\tilde Y})$, a compact metrizable group $K$, a compact subgroup $L$ of $K$, and an ergodic $\ConcProb_\Gamma$-homogeneous extension $({\tilde X}, \mu_{\tilde X}, T_{\tilde X}, \pi_{\tilde X \to \tilde Y}, \tilde \theta, \tilde \rho)$ of $(\tilde Y,T_{\tilde Y})$ by $K/L$.  Casting to $\ProbAlgG$, we obtain an ergodic $\ProbAlgG$-homogeneous extension  $(X, \mu_{X}, T_{X}, \pi_{X \to Y}, \theta, \rho)$ of the $\ProbAlgG$-system $(Y,T_Y) = (\tilde Y, T_{\tilde Y})_{\ProbAlgG}$.  Applying Theorem \ref{mackey-uncountable}(ii), we see that $(X, \mu_{X}, T_{X}, \pi_{X \to Y}, \theta, \rho)$ is equivalent in $\ProbAlgG$ to a $\ProbAlgG$-homogeneous skew-product $Y \rtimes^{\ProbAlgG}_{\rho_*} H/M$ for some compact subgroup $H$ of $K$, some compact subgroup $M$ of $H$, and some $H$-valued $\ProbAlgG$-cocycle $\rho_*$ on $Y$.  To finish the task, it suffices to locate an $H$-valued $\ConcProb_\Gamma$-cocycle $\tilde \rho_*$ on $\tilde Y$ such that $Y \rtimes^{\ProbAlgG}_{\rho_*} H/M$ is equivalent in $\ProbAlgG$ to $\tilde Y \rtimes^{\ConcProb_\Gamma}_{\tilde \rho_*} H/M$.

Let $\theta_* \in \Cond_X(H/M)$ be the vertical coordinate of  $Y \rtimes^{\ProbAlgG}_{\rho_*} H/M$ as viewed in the equivalent extension $(X, \mu_{X}, T_{X}, \pi_{X \to Y}, \theta, \rho)$.  From an inspection of the proof of Theorem \ref{mackey-uncountable}(ii), the cocycle $\rho_*$ is related to the cocycle $\rho$ by the formula
$$ \rho_{*,\gamma} = (\Phi_* \circ T_Y^\gamma) \rho_\gamma \Phi_*^{-1}$$
for some $\Phi \in \Cond_Y(H)$.  Since $H$ is compact metrizable by hypothesis (c), it is a Polish space, hence by \cite[Proposition 3.2]{jt19} we can find a $\ConcMes$-morphism $\tilde \Phi \colon \tilde Y \to H$ that models $\Phi \colon \Inc(Y) \to H$ in the sense that $\Phi =_\AbsMes \tilde \Phi \circ \iota$ where $\iota \colon \Inc(Y) \to \tilde Y$ is the canonical $\AbsProb$-inclusion.  If we then define
$$ \tilde \rho_{*,\gamma} = (\tilde \Phi_* \circ T_{\tilde Y}^\gamma) \tilde \rho_\gamma \tilde \Phi_*^{-1}$$
then we see that $\tilde \rho_* = (\tilde \rho_{*,\gamma})_{\gamma \in \Gamma}$ is an $H$-valued $\ConcProb_\Gamma$-valued cocycle on $\tilde Y$, which agrees with $\rho_*$ in the sense that
$$ 
\rho_{*,\gamma} =_\AbsMes \tilde \rho_{*,\gamma} \circ \iota.
$$

It remains to show that $Y \rtimes^{\ProbAlgG}_{\rho_*} H/M$ is equivalent in $\ProbAlgG$ to $\tilde Y \rtimes^{\ConcProb_\Gamma}_{\tilde \rho_*} H/M$.  By Proposition \ref{skew-crit}, it suffices to show that
\begin{align*}
 &\int_{\tilde Y \times^\ConcProb H/M} (f \circ \pi_{\tilde Y \times^\ConcProb H/M \to \tilde Y}) (g \circ \pi_{\tilde Y \times^\ConcProb H/M \to H/M}) \\
&= \left(\int_Y f\right) \left(\int_{H/M} g\right)
\end{align*}
for any $f \in L^\infty(\tilde Y) \equiv L^\infty(Y)$ and $g \in C(H/M)$; but this is immediate from Fubini's theorem.

\begin{remark} This argument shows that the hypotheses (a), (b) in Theorem \ref{mackey-countable} can in fact be deleted.  We do not know if the same is true for hypothesis (c).
\end{remark}

\section{A cocycle-free description of group extensions}\label{cocycle-free}

Let $\Gamma$ be a group, let $Y = (Y_\ProbAlg, T_Y)$ be a $\ProbAlgG$-system, let $K$ be a compact Hausdorff group, and let $\rho$ be a $K$-valued $\ProbAlgG$-cocycle on $Y$.  Then, as discussed in Section \ref{product-sec}, the skew-product 
$$ X := Y \rtimes_\rho K = Y \rtimes_\rho^{\ProbAlgG} K$$
can be promoted to a $\ProbAlgGK$-system.   From Lemma \ref{trans} it is not difficult to see that a function $f \in L^\infty(X)$ is $K^\op$-invariant if and only if it arises from a function in $L^\infty(Y)$; thus the invariant factor $\mathtt{Inv}_{K^\op}(X)$ of $X$ is $\ProbAlgG$-isomorphic to $Y$.  Furthermore, the action of $K^\op$, viewed as a homomorphism from $K^\op$ to the unitary group $U(L^2(X))$ of $L^2(X)$, can easily be seen to be continuous (giving the unitary group the strong operator topology) and faithful (i.e., injective), as can be easily seen by passing to the topological model $\Stone(Y) \rtimes_{\tilde \rho}^{\CHProb_\Gamma} K$ and using Urysohn's lemma.  

In this section we establish a converse to this observation (for ergodic systems), which was suggested to us by the anonymous referee, and which can be viewed as an abstract uncountable version of \cite[Theorem 3.29]{glasner2015ergodic}:

\begin{theorem}[Cocycle-free description of a group extension]\label{cocycle-free-thm}  Let $\Gamma$ be a group, $K$ be a compact Hausdorff group, and let $X$ be a $\ProbAlgGK$-system which is ergodic with respect to the $\Gamma$ action (thus $\Inv(X)$ is trivial).  Suppose that the action of $K^\op$ (viewed as a homomorphism from $K^\op$ to $U(L^2(X))$) is continuous and faithful, and let $Y := \mathtt{Inv}_{K^\op}(X)$ be the $K^\op$-invariant factor.  Then $X$ is $(\ProbAlgGK \downarrow Y)$-isomorphic to $Y \rtimes_\rho K$ for some $K$-valued $\ProbAlgG$-cocycle on $Y$.
\end{theorem}

\begin{proof}  In addition to the canonical model $\Stone(X)$ of $X$, it will also be convenient to work with a slightly smaller topological model for $X$ which we called the "Koopman model" in \cite[Appendix A.4]{jt21} (and which we learned from \cite[\S 19.3.1]{host2018nilpotent}).  Let ${\mathcal A}$ denote the algebra of functions $f \in L^\infty(X)$ which are \emph{$K^\op$-continuous} in the sense that the function $k^\op \mapsto k^\op f$ is continuous from $K^\op$ to $L^\infty(X)$.  This is easily seen to be a $\Gamma \times K^\op$-invariant unital $C^*$-algebra whose unit ball is dense (in the $L^2(X)$ topology) in $L^\infty(X)$, equipped with an invariant trace $f \mapsto \int_X f$, and hence its Gelfand dual\footnote{More specifically, we are using here the duality of categories between tracial unital $C^*$-algebras and $\CHProb$; see \cite[Theorem 5.11]{jt-foundational}.}  $\hat X$ is a $\CHProb_{\Gamma \times K^\op}$-system that is $\ProbAlgGK$-isomorphic to $X$, with a canonical identification ${\mathcal A} \equiv C(\hat X)$.  Since $L^\infty(Y)$ includes into ${\mathcal A}$, we have a $\CHProb_{\Gamma \times K^{op}}$-factor map  $\hat \pi \colon \hat X \to \Stone(Y)$, which is surjective since all non-trivial open sets in $\Stone(Y)$ have positive measure.  Because $K^\op$ acts continuously on ${\mathcal A}$, it also acts continuously on $\hat X$ (this is the main advantage in working with $\hat X$ in place of $\Stone(X)$).

From Lemma \ref{surj}(ii) and a theorem of Gleason \cite[Theorem 2.5]{gleason}, $\Stone(Y)$ is projective in $\CH$, and so we can find a $\CH$-section $s \colon \Stone(Y) \to \hat X$ of $\hat \pi$, thus $s$ is continuous and $\hat \pi \circ s(y) = y$ for all $y \in \Stone(Y)$.  In particular, for $y \in \Stone(Y)$, the orbit $K^\op s(y) = \{ k^\op s(y): k^\op \in K^\op \}$ is a compact subset of $\hat \pi^{-1}(y)$.  We claim that it is in fact all of $\hat \pi^{-1}(y)$ (thus the action of $K^\op$ is transitive on the fibers of $\hat \pi$).  For this we use an argument from \cite[\S 19.3.3, Lemma 10]{host2018nilpotent}.  If there was a point $x_0 \in \hat \pi^{-1}(y)$ outside of the $K^\op$-orbit of $s(y)$, then by Urysohn's lemma one can find a continuous non-negative function $f \in C(\hat X)$ which is positive at $s(y)$ and supported on some open set $U$ whose $K^\op$-orbit avoids $x_0$.  The averaged function 
$$ \overline{f} := \int_{K^\op} k^\op f\ d\Haar_{K^\op}(k^\op)$$
is then a $K^\op$-invariant function in $C(\hat X) \equiv {\mathcal A}$ which is positive at $s(y)$ but vanishes at $x_0$. But the only $K^\op$-invariant functions in ${\mathcal A}$ arise from $L^\infty(Y)$, hence $\overline{f}$ factors through $\hat \pi$ and thus attains the same value at both $s(y)$ and $x_0$, a contradiction.

If we knew that the action of $K^\op$ on $\hat X$ was free, one could now build a coordinate function $\tilde \theta \colon \hat X \to K$ by requiring $\tilde \theta(x)$ to be the unique element of $K$ such that
$$ x = \tilde \theta(x) s(\hat\pi(x))$$
for each $x \in \hat X$, and also build a cocycle $\tilde \rho_\gamma \colon \Stone(Y) \to K$ by requiring $\tilde \rho_\gamma(y)$ to be the unique element of $K$ such that
$$ T_{\hat X}^\gamma s(y) = \tilde \rho_\gamma(y) s(T_{\Stone(Y)}^\gamma y)$$
for each $y \in \Stone(Y)$ and $\gamma \in \Gamma$.  Unfortunately we are not quite able to establish the freeness of this action when $K$ is not metrizable, and must allow for the fact that for any given $x \in \hat X$, the stabilizer group
$$ K_x := \{ k \in K: k^\op x = x\}$$
could be non-trivial.  However, $K_x$ is always a closed subgroup of $K$, and it also obeys the invariance
\begin{equation}\label{K-inv}
K_{T_{\hat X}^{(\gamma,k_0^{\op})} x} = k_0^{-1} K_x k_0
\end{equation}
for all $\gamma \in \Gamma$, $k_0 \in K$, $x \in \hat X$.  In particular, the map $x \mapsto K_x$ is $\Gamma$-invariant.

The next step (as in \cite[Lemma 3.28]{glasner2015ergodic}) is to exploit the ergodicity of $X$, but first we need to work with metrizable quotients of $K$ in order to ensure that certain target spaces are Polish.  Let ${\mathcal N}$ denote the collection of all compact normal subgroups $N$ of $K$ such that $K/N$ is metrizable; as we shall discuss later, the Peter--Weyl theorem provides a plentiful supply of such subgroups.  If $N \in {\mathcal N}$, we let ${\mathcal G}_N$ denote the space of all closed subgroups of $K/N$; equipped with the Hausdorff metric, this is a Polish space.  Arguing exactly as in \cite[Lemma 3.28]{glasner2015ergodic}, we see that the map $x \mapsto K_x N/N$ is an upper semi-continuous map from $\hat X$ to the Polish space ${\mathcal G}_N$ and is thus Borel measurable.  At this point we encounter a minor technical difficulty in that the $\CHProb$-space $\hat X$ is equipped with the Baire $\sigma$-algebra rather than the Borel $\sigma$-algebra.  However, by \cite[Corollary 5.5]{jt-foundational}, the Baire-Radon measure $\mu_{\hat X}$ on the Baire $\sigma$-algebra can be uniquely extended to a Radon measure (which by abuse of notation we also call $\mu_{\hat X}$) on the Borel $\sigma$-algebra.  From the Radon property, $C(\hat X)$ is dense in $L^2(\hat X)$ in both $\sigma$-algebras, and hence every Borel set in $\hat X$ is equivalent up to $\mu_{\hat X}$-null sets to a Baire set.  In particular, $\hat X$ remains $\Gamma$-ergodic even with the Borel $\sigma$-algebra (as the Baire and Borel $\sigma$-algbras generate the same $\ProbAlgGK$-system).  The map $x \mapsto K_x N/N$ is a Borel measurable $\Gamma$-invariant map into a Polish space, and is thus constant almost everywhere by ergodicity (see \cite[Theorem 3.10.3]{glasner2015ergodic}).  Thus, for each $N$ there is a unique closed subgroup $K_N/N$ of $K/N$ such that $K_x N/N = K_N/N$ for $\mu_{\hat X}$-almost all $x \in \hat X$.  Comparing this with \eqref{K-inv} we see that $K_N/N$ must be normal in $K/N$, thus $K_N$ is a closed normal subgroup of $K$ containing $N$ and $K_x N = K_N$ for $\mu_{\hat X}$-almost all $x \in \hat X$.  In particular $K_N$ is monotone increasing in $N$: $K_N \leq K_{N'}$ whenever $N \leq N'$ lie in ${\mathcal N}$.

As a substitute for the action of $K$ being free, we claim that\footnote{This morally implies that $K_x$ is trivial for almost all $x$, but we cannot quite conclude this because ${\mathcal N}$ can be uncountably infinite, and the uncountable union of null sets need not be null.} $\bigcap_{N \in {\mathcal N}} K_N = \{1\}$.  Indeed, suppose for contradiction that there was a non-identity element $k$ of $K$ with $k \in K_N$ for all $N \in {\mathcal N}$.  Then for any $f \in C(\hat X)$ that is $N$-invariant for some $N \in {\mathcal N}$, we have $k^\op f = f$ $\mu_X$-almost everywhere, hence everywhere by continuity.  By the Peter--Weyl theorem (see e.g., \cite[Theorem 1.4.14]{tao-hilbert}), every neighborhood $U$ of the identity in $K$ contains an $N$ in ${\mathcal N}$.  By approximating an arbitrary element $f$ of $C(\hat X)$ by averages $\int_N k^{\op} f\ d\Haar_N(k)$ and using uniform continuity, we thus see that the space of functions in $C(\hat X)$ that are $N$-invariant for some $N \in {\mathcal N}$ is dense in $C(\hat X)$.  Thus $k^\op$ acts trivially on all of $C(\hat X)$, and hence on $L^\infty(X)$ by density, contradicting the assumption that the action of $K$ is faithful.  This establishes the claim $\bigcap_{N \in {\mathcal N}} K_N = \{1\}$.

For each $N \in {\mathcal N}$, the set $\{x: K_x N = K_N \}$ has full measure in $\hat X$, and is also open by (semi-)continuity (with $K_x N \geq K_N$ for all $x$); by \eqref{K-inv} we see that this set is also $\Gamma \times K^\op$-invariant, and thus arises from an open full measure $\Gamma$-invariant subset $\Stone(Y)_N$ of $\Stone(Y)$.  For any $x \in \hat \pi^{-1}(\Stone(Y)_N)$, we have $K_x N = K_N$ and the action of $K^\op$ on $\hat \pi^{-1}(\hat \pi(x))$ is transitive, there is thus a unique $\hat \theta_N(x) \in K/K_N$ such that
\begin{equation}\label{x1}
 K_N^\op x = K_N^\op \hat \theta_N(x) s(\hat \pi(x)),
 \end{equation}
and similarly for any $y \in \Stone(Y)_N$ and $\gamma \in \Gamma$, there is a unique $\hat \rho_{\gamma,N}(y) \in K/K_N$ such that
\begin{equation}\label{x2}
K_N^\op T^\gamma_{\hat X} s(y) = K_N^\op \hat \rho_{\gamma,N}(y) s(T^\gamma_{\Stone(Y)} y).
\end{equation}
It is a routine matter to verify that $\hat \theta_N$ is continuous on $\hat \pi^{-1}(\Stone(Y)_N)$ and $\hat \rho_{\gamma,N}$ is continuous on $\Stone(Y)_N$; in particular they are measurable up to almost everywhere equivalence.  From construction it is a routine matter to verify the identities
$$ \hat \theta_N(k^\op x) = \hat \theta_N(x) k$$
for all $x \in \hat \pi^{-1}(\Stone(Y)_N)$ and $k \in K$; by evaluating $T^{\gamma \gamma'}_{\hat X} s(y)$ in two different ways one also arrives at the identity
$$ \hat \rho_{\gamma \gamma',N}(y) = \hat \rho_\gamma(T_{\Stone(Y)}^{\gamma'} y) \hat \rho_{\gamma'}(y)$$
for all $y \in \Stone(Y)_N$ and $\gamma,\gamma' \in \Gamma$.  From applying $T^\gamma_{\hat X}$ to \eqref{x1} and using \eqref{x1}, \eqref{x2} one also has
$$ \hat \theta_N( T^\gamma_{\hat X} x ) = \hat \rho_{\gamma,N}(\hat \pi(x)) \hat \theta_N(x)$$
for all $x \in \hat \pi^{-1}(\Stone(Y)_N)$ and $\gamma \in \Gamma$.  Casting to $\ProbAlg$, we obtain $\ProbAlg$-morphisms $\theta_N \colon X \to K/K_N$ and $\rho_{\gamma,N} \colon Y \to K/K_N$ such that
$$ k^\op \theta_N = \theta_N k^{-1}$$
for all $k \in K$ and
$$ \rho_{\gamma \gamma',N} = (\rho_{\gamma,N} \circ \Inc(T_Y^{\gamma'})) \rho_{\gamma',N}$$
for all $\gamma,\gamma' \in \Gamma$, and also
$$ \theta_N \circ \Inc(T^\gamma_X) = (\rho_{\gamma,N} \circ \Inc(\pi)) \theta_N$$
for all $\gamma \in \Gamma$, where $\pi \colon X \to Y$ is the projection map.

Finally, we need to eliminate the quotienting by $K_N$.  Here it is convenient to return to the canonical model $\Stone(X)$ of $X$ rather than the Koopman model $\hat X$ as it allows us to avoid having to handle exceptional null sets.  If we let $\tilde \rho_{\gamma,N} \colon \Stone(Y) \to K/K_N$ be the canonical representative of $\rho_{\gamma,N}$ given by Theorem \ref{canon}(iv), and similarly let $\tilde \theta_N \colon \Stone(X) \to K/K_N$ be the canonical representative of $\theta_N$, then $\tilde \rho_{\gamma,N}$ is a $K/K_N$-valued $\CHProb_\Gamma$-cocycle on $\Stone(Y)$ with
$$ \tilde \theta_N(k^\op x) = \tilde \theta_N(x) k$$
for all $x \in \Stone(X)$ and $k \in K$, and
$$ \tilde \theta_N( T^\gamma_{\Stone(X)} x ) = \tilde \rho_{\gamma,N}(\Stone(\pi)(x)) \tilde \theta_N(x)$$
for all $x \in \Stone(X)$ and $\gamma \in \Gamma$.  Also from construction we see that whenever $N \leq N'$ both lie in ${\mathcal N}$, then the projection of $\hat \theta_{N'}$ to $K/K_N$ agrees almost everywhere with $\hat \theta_N$, hence on casting to $\ProbAlg$ and applying $\Stone$ we see that the projection of $\tilde \theta_{N'}$ to $K/K_N$ agrees everywhere with $\tilde \theta_N$.  Similarly the projection of $\tilde \rho_{\gamma,N'}$ to $K/K_N$ agrees everywhere with $\tilde \rho_{\gamma,N}$.  Since $\bigcap_{N \in {\mathcal N}} K_N = \{1\}$ (and also ${\mathcal N}$ is closed under finite intersections), one can then pass to an inverse limit and obtain a unique $\CH$-morphism $\tilde \theta \colon \Stone(X) \to K$ and a unique $K$-valued $\CHProb_\Gamma$-cocycle $\tilde \rho$ on $\Stone(Y)$ such that
\begin{equation}\label{thx}
 \tilde \theta(k^\op x) = \tilde \theta(x) k
 \end{equation}
for all $x \in \Stone(X)$ and $k \in K$, and
$$ \tilde \theta( T^\gamma_{\Stone(X)} x ) = \tilde \rho_{\gamma}(\Stone(\pi)(x)) \tilde \theta_N(x)$$
for all $x \in \Stone(X)$ and $\gamma \in \Gamma$.  Indeed each $\tilde \theta_N$ becomes the projection of $\tilde \theta$ to $K/K_N$, and similarly for the $\tilde \rho_{\gamma,N}$.

Let $\theta \in \Cond(K)$ be the abstraction of $\tilde \theta$.  We now claim that $\theta, \Inc(\pi)$ generate the $\sigma$-complete Boolean algebra $\Inc(X)_\Bool$, thus making $(X,\pi,\theta,\rho)$ a $\ProbAlgG$-group extension of $Y$ by $K$.  To see this, we return to the Koopman model $\hat X$ and observe from the Stone--Weierstra{\ss} theorem that finite linear combinations of functions of the form
$$ x \mapsto f(\hat \pi(x)) g(\theta_N(x))$$
for $f \in C(\Stone(Y))$, $N \in {\mathcal N}$, and $g \in C(K/K_N)$, form a unital algebra in $C(\hat X)$ that separates points, and is hence dense in $C(\hat X)$.  All of these functions, when casted to $\ProbAlg$, are measurable with respect to the $\sigma$-complete Boolean algebra generated by $\theta, \Inc(\pi)$, and hence $L^\infty$ of this algebra is dense (in $L^2(\hat X)$) in $C(\hat X)$, and hence in $L^2(\hat X) \equiv L^2(X)$, giving the claim.

To conclude, it suffices by Proposition \ref{skew-crit} to verify the identity
\begin{equation}\label{verify}
 \int_X (\pi^* f) (g \circ \theta)\ d\mu_X = \left( \int_Y f\ d\mu_Y \right) \left( \int_K g\ d\Haar_K \right)
 \end{equation}
 whenever $f \in L^\infty(Y)$ and $g \in C(K)$.  When $g=1$ the claim is clear, so we may assume $g$ has mean zero, so that the right-hand side of \eqref{verify} vanishes.  Using the $K^\op$-invariance of the measure $\mu_X$ and \eqref{thx} we can write the left-hand side as
 $$ \int_K \left( \int_X (\pi^* f) (k^\op g \circ \theta)\ d\mu_X\right)\ d\Haar_K(k)$$
 where $k^\op g(k_0) := g(k_0 k)$.  But $\int_K k^\op g \circ \theta\ d\Haar_K(k)$ vanishes, and the claim now follows from Fubini's theorem.
\end{proof}

\begin{remark} If one drops the requirement that the action of $K^\op$ be faithful in Theorem \ref{cocycle-free-thm}, then the same arguments give a variant of the conclusion in which $Y \rtimes_\rho K$ is replaced by $Y \rtimes_\rho K/N$ for some compact normal subgroup $N$ of $K$ and some $K/N$-valued $\ProbAlgG$-cocycle $\rho$.  Indeed one can just take $N$ to be the kernel of the action, and then one can apply Theorem \ref{cocycle-free-thm} to the quotiented action.
\end{remark}

\appendix 

\section{The equivalence of the uncountable topological and ergodic-theoretic Mackey--Zimmer theorems} \label{appendix}

In this appendix, we prove the equivalence of Theorem \ref{mackey-uncountable}(i) and the uncountable topological  Mackey--Zimmer theorem established  in \cite[\S 4]{ellis} by Ellis.  
First we extract from \cite[\S 4]{ellis} a suitable version and translate this statement to our language. 

\begin{theorem}[Ellis topological Mackey--Zimmer theorem]\label{thm-ellis}
Let $\Gamma$ be a discrete group, $K$ a compact Hausdorff group and $(Y,\mu_Y,T_Y)$ be an ergodic $\ProbAlgG$-system.  
Let $(X,R_X)$ be a $\CH_K$-promotion and $(X,T_X)$ be a $\CH_\Gamma$-promotion of a $\CH$-space $X$ respectively satisfying the following properties:
\begin{itemize}
    \item[(i)] $R_X$ acts freely on $X$.  
    \item[(ii)] $T^\gamma_X \circ R^k_X = R^k_X \circ T^\gamma_X$ for all $k\in K$,  $\gamma\in \Gamma$. 
    \item[(iii)] The orbit space\footnote{The orbit space of $(X,R_X)$ is the $\CH$-space $X/K$ resulting from the orbit equivalence relation $x\sim x'$ whenever there is $k\in K$ with $x=R^k_X(x')$. We denote by $\Pi:X\to X/K$ the canonical $\CH$-epimorphism. Note that the $\Gamma$-action $T_X$ naturally restricts   to a $\Gamma$-action $T_{X/K}$ on $X/K$ by (ii).} $(X/K,T_{X/K})$ is $\CH_\Gamma$-isomorphic to  $(\Stone(Y),T_{\Stone(Y)})$.  
\end{itemize}
Suppose $\Pi:(X,\mu_X,T_X)\to (\Stone(Y),\mu_{\Stone(Y)},T_{\Stone(Y)})$ is an extension of ergodic  $\CHProb_\Gamma$-systems. 
Then there exist a closed subgroup $H$ of $K$ and a cocycle $\rho=(\rho_\gamma)_{\gamma\in\Gamma}$,  $\rho_\gamma\in \CondTilde_{Y}(H)$ 
such that:
\begin{itemize}
    \item[(i)] The canonical map $\Pi:X/H\to \Stone(Y)$ is a $\CH_\Gamma$-isomorphism, and there is no proper closed subgroup $H'\leq H$ such that $\Pi:X/H'\to \Stone(Y)$ is a $\CH_\Gamma$-isomorphism.  
    \item[(ii)] $(X,T_X)$ is $\CH_\Gamma$-isomorphic to $(\Stone(Y)\rtimes_\rho K,T_{\Stone(Y)\rtimes_\rho K})$ where $$T_{\Stone(Y)\rtimes_\rho K}^\gamma(y,k):=(T_{\Stone(Y)}(y),\rho_\gamma(y)k).$$
    \item[(iii)] $\mu_X$ is the natural lift of the product measure $\mu_{\Stone(Y)}\times \Haar_H$ from $\Stone(Y)\times^\CH H$ to $X$.  
\end{itemize}
\end{theorem}

\begin{proof}
First use \cite[Theorem 4.9]{ellis} to find a continuous cocycle $\rho'=(\rho'_\gamma)_{\gamma\in\Gamma}$ on $\Stone(Y)$ with values in $K$ such that $(X,T_X)$ is $\CH_\Gamma$-isomorphic to $(\Stone(Y)\rtimes_{\rho'}K, T_{\Stone(Y)\rtimes_{\rho'}K})$ where $T^\gamma_{\Stone(Y)\rtimes_{\rho'}K}(y,k)=(T_{\Stone(Y)}(y), \rho'_\gamma(y) k)$. 
Second apply \cite[Theorem 4.15]{ellis} in order to extract a minimal cohomologous cocycle $\rho=(\rho_\gamma)_{\gamma\in\Gamma}$ on $\Stone(Y)$ with values in $H$ (where $H$ satisfies the irreducibility property in part (i) of the conclusions of Theorem \ref{thm-ellis}), which yields (ii).  
Now we lift $\mu_{\Stone(Y)}\times \Haar_H$ to $\Stone(Y)\times^\CH K$ by restriction, and adopting a similar argument to the one in the beginning of the proof of \cite[Theorem 4.16]{ellis} (which in turn relies on non-trivial Choquet theoretic results developed in \cite{keynes-newton}), we can identify $\mu_X$ with this ergodic lift, giving (iii). 
\end{proof}

\begin{proposition}
Theorem \ref{thm-ellis} is equivalent to part (i) of Theorem \ref{mackey-uncountable}.   
\end{proposition}

\begin{proof}
We show how \ref{thm-ellis} implies Theorem \ref{mackey-uncountable}(i). 
Let $Y = (Y_\ProbAlg, T_Y)$ be a $\ProbAlgG$-system, and let $K$ be a compact Hausdorff group. Let $(X, \pi, \theta, \rho)$ be an ergodic $\ProbAlgG$-group extension of $Y$ by $K$. 
We apply the canonical model functor $\Stone$ to obtain the $\CHProb_\Gamma$-extension $(\Stone(X),\Stone(\pi),\tilde{\theta},\tilde{\rho})$ of $\Stone(Y)$ by $K$, where $\tilde{\rho}$ is the functorial concrete representation of $\rho$ as a $K$-valued $\CHProb_\Gamma$-cocycle on $\Stone(Y)$, and $\tilde{\theta}\in \CondTilde_X(K)$ is the continuous representative of $\theta\in \Cond_X(K)$. 
As in the proof of Proposition \ref{skew-crit}, we let $\tilde{\phi}:\Stone(X)\to \Stone(Y)\times^\CH K$ be the $\CH$-morphism defined by $\tilde{\phi}(\tilde{x}):=(\Stone(\pi)(\tilde{x}),\tilde{\theta}(\tilde{x}))$ for $\tilde{x}\in \Stone(X)_\Set$. From $\tilde{\theta}\circ T^\gamma_{\Stone(X)}=(\tilde{\rho}_\gamma\circ \Stone(\pi))\tilde{\theta}$ (which results from \eqref{theta-gam-abstract} after applying the concrete model functor), we see that $\tilde{\phi}$ can be promoted to a $\CH_\Gamma$-morphism of extensions from $(\Stone(X),\Stone(\pi))$ to $(\Stone(Y)\rtimes_{\tilde{\rho}} K, \pi_{\Stone(Y)})$ where $\pi_{\Stone(Y)}:\Stone(Y)\times K\to \Stone(Y)$ is the canonical $\CH_\Gamma$-morphism.  We can use $\tilde{\phi}$ to pushforward $\mu_{\Stone(X)}$ from $\Stone(X)$ to $\Stone(Y)\times^\CH K$. Thus $\tilde{\phi}$ can be promoted to a $\CHProb_\Gamma$-morphism of extensions from $(\Stone(X),\Stone(\pi))$ to $(\Stone(Y)\rtimes_{\tilde{\rho}} K, \pi_{\Stone(Y)})$ if we equip $\Stone(Y)\rtimes_{\tilde{\rho}} K$ with the pushforward probability measure $\mu_{\Stone(X)}':=\tilde{\phi}_*\mu_{\Stone(X)}$.  
Arguing similarly to the last part of the proof of Proposition \ref{skew-crit}, one can show that the  $\ProbAlgG$-morphism of extensions from $(X,\pi)$ to $(\Stone(Y)\rtimes_{\tilde{\rho}} K, \pi_{\Stone(Y)})_{\ProbAlgG}$ is in fact a $\ProbAlgG$-isomorphism of extensions.  
Now it is not difficult to check that the $\CHProb_\Gamma$-extension $(\Stone(Y)\rtimes_{\tilde{\rho}} K, \pi_{\Stone(Y)})$ satisfies the assumptions in Theorem \ref{thm-ellis}. In particular, we find that $(X,\mu_X,T_X)$ is $\ProbAlgG$-isomorphic to $(\Stone(Y)\rtimes_{\rho'} H)_{\ProbAlgG}$ for some continuous cocycle $\rho'$ on $\Stone(Y)$ with values in $H$ that is cohomologous to $\tilde{\rho}$.   

Conversely, let $\Pi:(X,\mu_X,T_X)\to (\Stone(Y),\mu_{\Stone(Y)},T_{\Stone(Y)})$ be an extension of ergodic  $\CHProb_\Gamma$-systems (with the conventions as in Theorem \ref{thm-ellis}).  By \cite[Theorem 4.9]{ellis}, we can identify $(X,T_X)$ with $(\Stone(Y)\rtimes_\rho K,T_{\Stone(Y)\rtimes_\rho K})$ as $\CH_\Gamma$-systems where $\rho=(\rho_\gamma)_{\gamma\in\Gamma}$ is a continuous cocycle on $\Stone(Y)$ with values in $K$.  Now the cast of  $\Pi:(\Stone(Y)\rtimes_\rho K,\mu_X,T_{\Stone(Y)\rtimes_\rho K})\to (\Stone(Y),\mu_{\Stone(Y)},T_{\Stone(Y)})$ to $\ProbAlgG$ satisfies the assumptions of Theorem \ref{mackey-uncountable}(i). Passing to concrete models and applying Proposition \ref{skew-crit}, we get the conclusions (ii), (iii), while for the irreducibility conclusion (i) we also need the equivalences in \cite[Theorem 4.12]{ellis}. 
\end{proof}

\subsection*{Acknowledgements}
AJ was supported by DFG-research fellowship JA 2512/3-1. 
TT was supported by a Simons Investigator grant, the James and Carol Collins Chair, the Mathematical Analysis \& Application Research Fund Endowment, and by NSF grant DMS-1764034. We thank Balint Farkas and Markus Haase for helpful comments and references, and the anonymous referee for several useful corrections and suggestions.


\begin{thebibliography}{abbrv}

\normalsize
\baselineskip=17pt


\bibitem{ab40abstract}
L.~Alaoglu and G.~Birkhoff.
\newblock {\em General ergodic theorems}.
\newblock  Ann. of Math., 41:293--309, 1940.


\bibitem{DNP}
R.~Derndinger, G.~Palm, and R.~Nagel.
\newblock {\em Ergodic theory in the perspective of functional analysis}.
\newblock unpublished.

\bibitem{doob-ratio}
J.~L. Doob.
\newblock {\em A ratio operator limit theorem}.
\newblock  Z. Wahrscheinlichkeitstheorie und Verw. Gebiete, 1:288--294,
  1962/63.

\bibitem{EFHN}
T.~Eisner, B.~Farkas, M.~Haase, and R.~Nagel.
\newblock {\em Operator theoretic aspects of ergodic theory}, volume 272 of
  Graduate Texts in Mathematics.
\newblock Springer, Cham, 2015.

\bibitem{ellis}
R.~Ellis.
\newblock {\em Topological dynamics and ergodic theory}.
\newblock  Ergodic Theory Dynam. Systems, 7:25--47, 1987.

\bibitem{fremlinvol3}
D.~H. Fremlin.
\newblock {\em Measure theory. {V}ol. 3}.
\newblock Torres Fremlin, Colchester, 2004.
\newblock Measure algebras, Corrected second printing of the 2002 original.

\bibitem{furstenberg2014recurrence}
H.~Furstenberg.
\newblock {\em Recurrence in Ergodic Theory and Combinatorial Number Theory}.
\newblock Princeton Legacy Library. Princeton University Press, 2014.

\bibitem{garling}
D.~J.~H. Garling.
\newblock {\em A "short" proof of the {R}iesz representation theorem}.
\newblock  Proc. Cambridge Philos. Soc., 73:459--460, 1973.

\bibitem{glasner2015ergodic}
E.~Glasner.
\newblock {\em Ergodic Theory via Joinings}.
\newblock Mathematical Surveys and Monographs. American Mathematical Society,
  2015.

\bibitem{gleason}
A.~M. Gleason.
\newblock {\em Projective topological spaces}.
\newblock  Illinois J. Math., 2:482--489, 1958.

\bibitem{hartig}
D.~G. Hartig.
\newblock {\em The {R}iesz representation theorem revisited}.
\newblock  Amer. Math. Monthly, 90:277--280, 1983.

\bibitem{host2005nonconventional}
B.~Host and B.~Kra.
\newblock {\em Nonconventional ergodic averages and nilmanifolds}.
\newblock  Ann.~Math., 161:397--488, 2005.

\bibitem{host2018nilpotent}
B.~Host and B.~Kra.
\newblock {\em Nilpotent Structures in Ergodic Theory}.
\newblock Mathematical Surveys and Monographs. American Mathematical Society,
  2018.

\bibitem{jamneshan2019fz}
A.~Jamneshan.
\newblock {\em An uncountable Furstenberg-Zimmer structure theory}.
\newblock  arXiv:2103.17167, 2021.

\bibitem{jt21}
A.~Jamneshan, O.~Shalom, and T.~Tao.
\newblock {\em The structure of arbitrary Conze-Lesigne systems}.
\newblock  arXiv:2112.02056, 2021.

\bibitem{jt19}
A.~Jamneshan and T.~Tao.
\newblock {\em An uncountable Moore-Schmidt theorem}.
\newblock  Ergodic Theory Dynam. Systems, to appear. 

\bibitem{jt-foundational}
A.~Jamneshan and T.~Tao.
\newblock {\em Foundational aspects of uncountable measure theory: Gelfand duality,
  Riesz representation, canonical models, and canonical disintegration}.
\newblock  arXiv:2010.00681, 2020.

\bibitem{keynes-newton}
H.~B. Keynes and D.~Newton.
\newblock {\em The structure of ergodic measures for compact group extensions}.
\newblock  Israel J. Math., 18:363--389, 1974.

\bibitem{mackey}
G.~W. Mackey.
\newblock {\em Ergodic theory and virtual groups}.
\newblock  Math. Ann., 166:187--207, 1966.

\bibitem{maharam}
D.~Maharam.
\newblock {\em On a theorem of von {N}eumann}.
\newblock  Proc. Amer. Math. Soc., 9:987--994, 1958.

\bibitem{moriakov}
N.~Moriakov.
\newblock {\em Entropy and Kolmogorov complexity}.
\newblock PhD thesis, TU Delft Analysis, 2016.

\bibitem{nedoma1957note}
J.~Nedoma.
\newblock {\em Note on generalized random variables}.
\newblock In  Transactions of the First Prague Conference on Information
  Theory, Statistical Decision Functions, Random Processes (Liblice, 1956),
  Publishing House of the Czechoslovak Academy of Sciences, Prague, pages
  139--141, 1957.

\bibitem{segal}
I.~E. Segal.
\newblock {\em Equivalences of measure spaces}.
\newblock  Amer. J. Math., 73:275--313, 1951.

\bibitem{sunder}
V.~S. Sunder.
\newblock {\em The {R}iesz representation theorem}.
\newblock  Indian J. Pure Appl. Math., 39:467--481, 2008.

\bibitem{tao2014mackey}
T.~Tao.
\newblock {\em An abstract ergodic theorem, and the Mackey-Zimmer theorem}.
\newblock Blog article, June, 20 2014.

\bibitem{tao-hilbert}
T.~Tao.
\newblock {\em Hilbert's fifth problem and related topics}, volume 153 of {\em
  Graduate Studies in Mathematics}.
\newblock American Mathematical Society, Providence, RI, 2014.

\bibitem{varadarajan-riesz}
V.~S. Varadarajan.
\newblock {\em On a theorem of {F}. {R}iesz concerning the form of linear
  functionals}.
\newblock Fund. Math., 46:209--220, 1959.

\bibitem{voevodsky}
V.~Voevodsky.
\newblock {\em A categorical approach to the probability theory}.
\newblock unpublished, 2004.

\bibitem{zimmer1976ergodic}
R.~J. Zimmer.
\newblock {\em Ergodic actions with a generalized spectrum}.
\newblock Illinois J.~Math., 20:555--588, 1976.

\bibitem{zimmer1976extension}
R.~J. Zimmer.
\newblock {\em Extension of ergodic group actions}.
\newblock Illinois J.~Math., 20:373--409, 1976.


\end{thebibliography}
\end{document}